\documentclass[12pt]{article}

\usepackage{utopia}

\usepackage{epsfig}
\usepackage{amsmath}
\usepackage{amssymb}
\usepackage{amsfonts}
\usepackage{amsthm}
\usepackage{mathrsfs}
\usepackage{hyperref}
\usepackage{algorithm}
\usepackage{algorithmic}
\usepackage{blindtext}
\usepackage[utf8]{inputenc}
\usepackage{graphicx}
\usepackage{subfigure}
\usepackage[usenames, dvipsnames]{color}
\usepackage{caption}
\usepackage{booktabs}
\usepackage{varwidth}

\newcommand {\C}        {{\mathbb{C}}}
\newcommand {\R}        {{\mathbb{R}}}
\newcommand {\N}        {{\mathbb{N}}}

\newcommand{\norm}[1]{\left\Vert#1\right\Vert}
\newcommand{\abs}[1]{\left\vert#1\right\vert}
\newcommand {\ri}	{{\mathrm i}}

 \newtheorem{theorem}{Theorem}[section]
 \newtheorem{lemma}[theorem]{Lemma}

 \numberwithin{equation}{section}

\title{Nonsmooth Rate-of-Convergence Analyses of Algorithms for Eigenvalue Optimization}
\author{Fatih~Kangal \and Emre~Mengi}

\date{\today}

\begin{document}
\maketitle

\begin{abstract}
\noindent
Non-smoothness at optimal points is a common phenomenon in many eigenvalue
optimization problems. We consider two recent algorithms to minimize the
largest eigenvalue of a Hermitian matrix dependent on one 
parameter, both proven to be globally convergent unaffected by
non-smoothness. One of these models the eigenvalue function with a piece-wise 
quadratic function, and effective in dealing with non-convex problems. The other projects
the Hermitian matrix into subspaces formed of eigenvectors, and effective in dealing
with large-scale problems. We generalize the latter slightly to cope with non-smoothness.
For both algorithms, we analyze the rate-of-convergence
in the non-smooth setting, when the largest eigenvalue is multiple at the minimizer
and zero is strictly in the interior of the generalized Clarke derivative, and prove that both algorithms
converge rapidly. The algorithms are applied to, and the deduced results are illustrated on 
the computation of the inner numerical radius, the modulus of the point on the boundary
of the field of values closest to the origin, which carries significance for instance for the 
numerical solution of a definite generalized symmetric eigenvalue problem.
\\[10pt]
\\[5pt]
\end{abstract}

\section{Introduction}
A pair of Hermitian matrices $A,B\in\C^{n\times n}$ is said to be definite if
\begin{equation}\label{definiteness}
	\begin{split}
			\gamma(A,B)\; &:= \; \min_{z\in\C^n, \norm{z}_2=1}\sqrt{(z^*Az)^2\; + \; (z^*Bz)^2} \\
					       &=  \; \min_{z\in\C^n, \norm{z}_2=1}\abs{z^*(A+ \ri B)z}	\\
					   	&= \;\;\;\;\;\: \min\left\{\abs{w} \; | \; w\in F(A+ \ri B)\right\} \; > \;0,
	\end{split}
\end{equation}
where $F(C)$ denotes the field of values of $C \in {\mathbb C}^{n\times n}$, the subset of ${\mathbb C}$
defined by
\[
F(C) \; := \; \left\{z^*Cz \in {\mathbb C} \; | \; z\in\C^n, \; \norm{z}_2=1\right\}.
\]


The definiteness of $(A,B)$ carries significance for the numerical solution of the generalized eigenvalue 
problem $Ax=\lambda Bx$ \cite{CHENG1999, Davies2001}. If $(A,B)$ is known to be definite, then the generalized 
eigenvalue problem can be transformed into another related problem $\widetilde{A} x  = \lambda \widetilde{B} x$ with 
$\lambda_{\min}(\widetilde{B}) = \gamma(A,B)$. Furthermore, the transformed problem can be solved
by calculating a Cholesky factorization $\widetilde{B} = R^\ast R$, and computing the eigenvalues
of the Hermitian matrix $R^{-\ast} \widetilde{A} R^{-1}$. This is a plausible procedure
with a small backward error provided $\gamma(A,B)$ is not small. 
Two other application areas concern Hermitian quadratic eigenvalue problems, in particular checking 
the hyperbolicity of such problems \cite{higham2002detecting}, and saddle point linear systems with 
symmetric indefinite coefficient matrices, in particular setting up a conjugate gradient iteration for 
such systems \cite{liesen2008nonsymmetric}.

Motivated by such applications, Cheng and Higham have focused on procedures for efficient
determination of whether a given pair $(A,B)$ of Hermitian matrices is definite or not \cite{CHENG1999}.
If the Hermitian pair is not definite, in the same paper, the authors have also considered the 
computation of a pair $(A+\Delta A_\ast, B + \Delta B_\ast)$ where $(\Delta A_\ast, \Delta B_\ast)$ 
solves the following minimization problem for a prescribed positive real number $\delta$:
\begin{equation}\label{eq:distance_definiteness}
	d_{\delta}(A,B)		:=	
			\min
				\left\{
					\left\|  
						\left[
							\begin{array}{cc}
								\Delta A & \Delta B
							\end{array}
						\right]  
					\right\|_2 \; \bigg| \; \gamma(A+\Delta A, B+\Delta B) \; \geq \; \delta  
				\right\}.
\end{equation}
The non-convex eigenvalue optimization problem
\begin{equation}\label{Main_problem}
	\min_{\theta \in [0,2\pi)} \: \lambda_{\max}( A \cos \theta + B \sin \theta),
\end{equation}
has shown to be closely related to both the determination of whether the pair $(A,B)$ is definite or not,
as well as the computation of an optimal $(\Delta A_\ast, \Delta B_\ast)$ such that 
$\gamma(A+\Delta A_\ast, B+\Delta B_\ast) \geq \delta$ and 
$
	d_{\delta}(A,B)		=
						\left\|
						\left[
							\begin{array}{cc}
								\Delta A_\ast & \Delta B_\ast
							\end{array}
						\right]
						\right\|_2. 
$
Note that above $\lambda_{\max}(\cdot)$ represents the largest eigenvalue of its matrix argument,
which is a notation we adopt throughout this text.
The minimal value in (\ref{Main_problem}) has a geometric meaning in terms of $F(A + \ri B)$;
in absolute value it corresponds to the inner numerical radius of $A+ \ri B$, i.e., the modulus of the
point on the boundary of $F(A + \ri B)$ closest to the origin. It is argued in \cite{CHENG1999} that the 
major challenge for the estimation of $d_{\delta}(A,B)$ is the global solution of the optimization problem 
in (\ref{Main_problem}).

The optimization problem in (\ref{Main_problem}) is only a special instance of a family of 
eigenvalue optimization problems
\begin{equation}\label{eq:main_problem2}
	\hskip 10ex
	\min_{\omega \in \Omega} \: \lambda_{\max} ({\mathcal A}(\omega)),	\quad\quad 
				{\mathcal A}(\omega) \; := \; \sum_{j=1}^\kappa f_{j}(\omega) A_j,
\end{equation}
where $\Omega$ is an interval in ${\mathbb R}$, the matrices $A_1, \dots, A_\kappa  \in {\mathbb C}^{n\times n}$
are Hermitian, and the functions $f_1, \dots, f_\kappa : \underline{\Omega} \rightarrow {\mathbb R}$ are 
real analytic on their domain $\underline{\Omega}$, which is an open interval in ${\mathbb R}$ 
containing $\Omega$. An eigenvalue optimization problem of the form (\ref{eq:main_problem2})
is typically non-convex excluding the very special affine case ${\mathcal A}(\omega) = A_1 + \omega A_2$.

Recently we have developed general algorithms \cite{Mengi2014, kangal2018} that are in many cases effective
in solving non-convex eigenvalue optimization problems of the form (\ref{eq:main_problem2}) globally.
The former of these \cite{Mengi2014} employs piece-wise quadratic functions to model the objective 
eigenvalue function, and is meant for small- to medium-scale problems. The latter \cite{kangal2018} introduces 
a subspace framework to deal with problems when the size of ${\mathcal A}(\omega)$ is large. It repeatedly 
projects ${\mathcal A}(\omega)$ to small subspaces, and minimizes the largest eigenvalue of the resulting 
projected matrix-valued function.

Here we present the adaptations of the algorithms in \cite{Mengi2014, kangal2018} for the 
solution of (\ref{Main_problem}). This paves the way for efficient determination of whether a Hermitian pair $(A,B)$
is definite or not, as well as efficient computation of the distance $d_{\delta}(A,B)$ in (\ref{eq:distance_definiteness}) 
and a nearest definite pair $(\widetilde{A},\widetilde{B})$ such that $\gamma(\widetilde{A},\widetilde{B}) \geq \delta$.
The adaptation of the algorithm in \cite{Mengi2014} is guaranteed to converge to the global
minimizer of (\ref{Main_problem}), and performs well in practice on small- to medium-scale problems.
The subspace framework in \cite{kangal2018} extends the range of applicability to quite large Hermitian matrix pairs.

The main contribution of this work is on the theoretical side. Global convergence of the algorithms has already been 
established; \cite{kangal2018} is globally convergent for the family of optimization problems (\ref{eq:main_problem2})
provided the projected problems are solved globally, whereas the algorithm in \cite{Mengi2014} is globally convergent 
for problems of the form (\ref{eq:main_problem2}) provided a global lower bound is known on 
$\lambda_{\max}''({\mathcal A}(\omega))$ over all $\omega$ where $\lambda_{\max}({\mathcal A}(\omega))$ is 
differentiable. We are also quite informed about the rate-of-convergences of these algorithms in the smooth case 
when $\lambda_{\max}({\mathcal A}(\omega_\ast))$ is simple at a converged global minimizer $\omega_\ast$; 
the subspace framework \cite{kangal2018} converges at a superlinear rate \cite{kangal2018, Kressner2017} 
both in theory and in practice, while we observe that the algorithm in \cite{Mengi2014} converges at a linear rate
(even though a formal proof is open, numerical experiments indicate a linear convergence convincingly). 
However, little is known about the rate-of-convergences in the presence of non-smoothness when 
$\lambda_{\max}({\mathcal A}(\omega_\ast))$ is not simple. In this work, we analyze the rate-of-convergences
of the algorithm in \cite{kangal2018, Kressner2017} on the problems of the form (\ref{eq:main_problem2})
in the non-smooth setting. We assume $\lambda_{\max}({\mathcal A}(\omega_\ast))$
is multiple, moreover $0 \in {\rm Int} \: \partial \lambda_{\max}({\mathcal A}(\omega_\ast))$
(i.e., zero lies strictly in the interior of $\partial \lambda_{\max}({\mathcal A}(\omega_\ast))$),
where the generalized Clarke derivative $\partial f (\omega)$ at $\omega_\ast$ of a univariate 
function $f(\omega)$ that is differentiable almost everywhere excluding a set $\Gamma$ of measure zero
is given by \cite{Clarke1990}
\begin{equation}\label{eq:gen_derivative}
	\partial f(\omega_\ast)
			\;\;	:=	\;\;
	{\rm Co}
	\left\{
		\lim_{k\rightarrow \infty} f'(\omega^{(k)})	\;	\big|	\;	
								\omega^{(k)} \rightarrow \omega_\ast,
								\;\;
								\omega^{(k)} \notin \Gamma \: \forall k
	\right\}
\end{equation}
with ${\rm Co}(\cdot)$ denoting the convex hull. 
The condition $0 \in {\rm Int} \: \partial \lambda_{\max}({\mathcal A}(\omega_\ast))$ is equivalent
to $\lambda_{\max}({\mathcal A}(\omega))$ having non-zero left-hand and non-zero right-hand 
derivatives at $\omega_\ast$, which is generically the case. In our analysis for the subspace
framework, we keep an additional assumption, namely we assume $0 \notin {\rm bd} \: \partial \lambda_{j}({\mathcal A}(\omega_\ast))$
(i.e., zero is not on the boundary of $\partial \lambda_{j}({\mathcal A}(\omega_\ast))$)
for the $j$th largest eigenvalue $\lambda_{j}({\mathcal A}(\omega))$ if 
$\lambda_{j}({\mathcal A}(\omega_\ast))
				=$
	$\lambda_{\max}({\mathcal A}(\omega_\ast))$. 
This assumption also holds generically in the non-smooth setting, and amounts to requiring that
$\lambda_j({\mathcal A}(\omega))$ has non-zero left-hand and right-hand 
derivatives at $\omega_\ast$. We prove that if the maximum of the errors of the last two iterates 
of the algorithm in \cite{Mengi2014} is $h$, the error of the next iterate 
is $O(h^2)$. We also generalize the subspace framework in \cite{kangal2018} to cope with the non-smooth setting, 
and show rigorously that the iterates of the proposed generalized framework converge at a quadratic rate.

Our work is exposed in the following order. Background on definite pairs and on the distance $d_{\delta}(A,B)$
are summarized in Section \ref{sec:background}; in particular the links between these concepts, the eigenvalue 
optimization problem (\ref{Main_problem}) and the inner numerical radius of $A + \ri B$ have been discussed.
The crucial task is the solution of (\ref{Main_problem}), equivalently the computation of the inner numerical
radius of $A + \ri B$. In Section \ref{small_scale_algorithms}, two algorithms are spelled out 
to compute the inner numerical radii for small to medium size matrices, namely a level-set method
and the algorithm in \cite{Mengi2014} based on piece-wise quadratic support functions. The latter
is presented in the general scope of (\ref{eq:main_problem2}) pointing out how it can be adapted for
(\ref{Main_problem}). The remarkable contribution is a rate-of-convergence analysis in 
Section \ref{sec:support_rate_convergence} for the piece-wise quadratic support based algorithm 
in the non-smooth setting when $\lambda_{\max}({\mathcal A}(\omega))$ is multiple and not differentiable 
at a converged global minimizer. The analysis indicates a rapid convergence, surprisingly faster than 
the smooth case when $\lambda_{\max}({\mathcal A}(\omega))$ is simple at the minimizer. 
This is followed by Section \ref{sec:subspace} which is devoted to the subspace framework 
to deal with (\ref{eq:main_problem2}) when the Hermitian matrices $A_1, \dots, A_\kappa$ are
large; the proposed framework generalizes the basic one in \cite{kangal2018} taking into account
also the possible non-smoothness at the optimal point. It is in particular applicable to solve 
(\ref{Main_problem}) and to compute the inner numerical radius of $A + \ri B$ for large $A$, $B$. 
In Section \ref{sec:subspace_rate-of-convergence}, we establish a quadratic rate-of-convergence 
of the proposed subspace framework formally in the non-smooth case, which was left open by the 
previous works. Numerical examples at the ends of Sections \ref{small_scale_algorithms} and \ref{sec:subspace} 
illustrate the efficiency of the algorithms, and confirm that the rate-of-convergences established 
in theory are realized in practice.
 
\section{Background on Definite Pairs and Nearest Definite Pairs}\label{sec:background}
\subsection{Connections with Fields of Values}\label{sec:definite_bg}
Recall that a given Hermitian pair $(A,B)$ is definite if 
\begin{equation*}
	\begin{split}
		\gamma(A,B) & \; = \; \min_{z\in\C^n, \norm{z}=1}\sqrt{(z^*Az)^2\; + \; (z^*Bz)^2} \\
					& \; = \quad \;	  \min \{ |z| \; | \; z \in F(A + \ri B) \}	\;\;  >  \;\; 0.
	\end{split}
\end{equation*}
The quantity $\gamma(A,B)$ is called the Crawford number. If  a given Hermitian pair $(A,B)$ is known to be definite, 
then the following result \cite{CHENG1999,STEWART1979} is helpful for the numerical solution of the generalized 
eigenvalue problem $Ax=\lambda Bx$.
\begin{theorem}\label{theta_existence}
	Suppose that $(A,B)$ is a definite Hermitian pair. Letting
	\[
	A_{\theta} \; := \; A\cos\theta + B\sin\theta \quad \text{and} \quad B_{\theta} \; := \; -A\sin\theta + B\cos\theta,
	\]
	for $\theta \in [0,2\pi)$, and $\phi \in [0, 2\pi)$ be such that  
	$\gamma(A,B) e^{{\rm i} \phi} = \arg \min_{z \in F(A + {\rm i} B)} \: |z|$, the following assertions hold:
	\begin{enumerate}
		\item[\bf (i)] $A_{\theta}+ \ri B_{\theta} = e^{-\ri \theta}(A+ \ri B)$.
		\item[\bf (ii)] The pair $(u_\theta,v_\theta)$ is an eigenvalue of $(A_\theta,B_\theta)$ 
		(i.e., ${\rm det}(v_\theta A_\theta - u_\theta B_\theta)=0$) if and only if the pair $(u, v)$ defined by		
		\[
			\left[
			\begin{array}{c}
				v \\ 
				u
			\end{array}
			\right]
			:=
			\left[
			\begin{array}{cc}
				\cos\theta & \sin\theta \\
				-\sin\theta & \cos\theta
			\end{array}
			\right]
			\left[
			\begin{array}{c}
				v_\theta \\
				u_\theta
			\end{array}
			\right]
		\] 
		is an eigenvalue of $(A,B)$ (i.e., ${\rm det}(vA - uB)=0$).
		\item[\bf (iii)] The matrix $B_{\varphi}$ is positive definite, indeed 
		$\lambda_{\min}(B_{\varphi}) = \gamma(A,B)$, where $\varphi := -\pi/2 + \phi$.
		\item[\bf (iv)] The inclusions 
			$F(A_{\varphi} + \ri B_{\varphi}) \; \subseteq \; H_{\varphi} := \{ \omega \in {\mathbb C} \; | \; {\rm Im}(\omega) \geq \gamma(A,B) \}$,
		and $\; \ri \gamma(A,B) \in F(A_{\varphi} + \ri B_{\varphi})$ hold.
	\end{enumerate}
\end{theorem}

The definiteness of a pair $(A,B)$ can be inferred from the inner numerical radius of $A + \ri B$,
where, for a matrix $C\in\C^{n\times n}$, its inner numerical radius is defined by
\[
	\zeta(C) \; :=\; \min\left\{\abs{z} \; | \; z \text{ is on the boundary of }F(C)\right\}.
\]
It follows from (\ref{definiteness}) that the pair $(A,B)$ is definite if and only if $0\notin F(A+ \ri B)$ 
if and only if $\gamma(A,B)=\zeta(A+ \ri B)$. The next result describes a concrete approach to
determine the definiteness of $(A,B)$ by solving an eigenvalue optimization problem associated
with $\zeta(A + \ri B)$.
\begin{theorem}[Cheng $\&$ Higham \cite{CHENG1999}]\label{char_inner}
	The inner numerical radius satisfies
	\begin{equation}\label{eq:eigopt_char}
		\zeta(C) \; = \; \abs{\min_{\theta \in [0,2\pi)}\lambda_{\max}(H(\theta))}
	\end{equation}
	where $H(\theta)=(Ce^{-\ri \theta} +C^\ast e^{\ri \theta})/2$. Letting
	$\; \theta_\ast := \arg\min_{\theta \in [0,2\pi)}\lambda_{\max}(H(\theta))$, 
	we furthermore have
	\begin{enumerate}
		\item[\bf (i)]  $0\in F(C) $ if and only if $\lambda_{\max}(H(\theta_\ast))\geq0$, and 
		\item[\bf (ii)] the point $\zeta(C)e^{i\phi}$ is on the boundary of $F(C)$ with 
		\begin{equation*}
				\phi
					\; := \;
				\left\{
					\begin{split}
						\; \theta_\ast \;\;\;\;\;		\quad\quad	&	{\rm if} \;\; 0\in F(C) 	\\
						\; \theta_\ast + \pi		\quad\quad 	&	{\rm if} \;\; 0\notin F(C)	
					\end{split}
				\right.
		\end{equation*}
	\end{enumerate}
\end{theorem}
\noindent
The eigenvalue optimization characterization (\ref{eq:eigopt_char}) is a consequence of the facts that the inner
numerical radius of $C$ and $Ce^{-\ri \theta}$ are the same, since $F(Ce^{-\ri \theta}) = e^{-\ri \theta} F(C)$,
as well as
\[
	\lambda_{\min}
		\left(
			\frac{Ce^{-\ri \theta} + C^\ast e^{\ri \theta}}{2}
		\right) 
			\;	\leq	\;
				{\rm Re} (z)
			\;	\leq	\;
	\lambda_{\max}
		\left(
			\frac{C e^{-\ri \theta} + C^\ast e^{\ri \theta}}{2}
		\right) 
	\quad
	\forall z \in F(Ce^{-\ri \theta})
\]
where the bounds are attained at points on the boundary of $F(Ce^{-\ri \theta})$.

Next suppose that $(A,B)$ is not definite. For a prescribed real number $\delta > 0$, we 
consider $d_\delta(A,B)$ defined as in (\ref{eq:distance_definiteness}), which corresponds
to a distance from $(A,B)$ to a nearest definite pair that is $\delta$-away from indefiniteness.
It is clear from this definition that if $(A,B)$ is definite with $\gamma(A,B) > \delta$, then we have $d_{\delta}(A,B)=0$.
Moreover, $d_\delta(A,B)$ has a characterization in terms of $\zeta(A + \ri B)$, which is stated next \cite{CHENG1999}.
\begin{theorem}\label{solving_distance}
	For a given Hermitian pair $(A,B)$, let $\; C := A + \ri B \;$ and 
	$\; \theta_\ast :=  \arg\min_{\theta \in [0,2\pi)}\lambda_{\max}(H(\theta))$,
	where 
	\[
		H(\theta)=A\cos\theta+B\sin\theta = (Ce^{-{\rm i} \theta} + C^\ast e^{{\rm i} \theta})/2.
	\] 
	\begin{enumerate}
		\item[\bf (i)]
		We have
		\begin{itemize}
			\item $d_{\delta}(A,B)=\delta+\zeta(C)$ if $\; 0\in F(C)$ \\
			(equivalently if $\; \lambda_{\max}(H(\theta_\ast))\geq0$), or 
			\item $d_{\delta}(A,B)=\max\{\delta-\zeta(C),0\}$ if $\; 0\notin F(C)$ \\
			(equivalently if $\; \lambda_{\max}(H(\theta_\ast)) < 0$). 
		\end{itemize}
		\item[\bf (ii)]
		Furthermore, letting $H(\theta_\ast)=Q{\rm diag}(\lambda_i)Q^\ast$ with $\lambda_n\leq\dots\leq\lambda_1$ 
		be a spectral decomposition, two sets of optimal perturbations in both cases are given by
		\begin{equation*}
			\begin{split}
			\Delta A_1 & =\cos\theta_\ast Q {\rm diag}(\min\{-\delta-\lambda_i,0\})Q^\ast,	\\
			\Delta B_1 & =\sin\theta_\ast Q {\rm diag}(\min\{-\delta-\lambda_i,0\})Q^\ast
			\end{split}
		\end{equation*}
		and
		\begin{equation*}
			\Delta A_2=-d_{\delta}(A,B)\cos\theta_*I, \quad \Delta B_2=-d_{\delta}(A,B)\sin\theta_*I.
		\end{equation*}
		\item[\bf (iii)] The inner numerical radius $\zeta(A+\Delta A_j + \ri (B+\Delta B_j))$ is attained at
		$\zeta(A+\Delta A_j + \ri (B+\Delta B_j)) e^{\ri(\theta_\ast + \pi)}$ for $j = 1,2$.
	\end{enumerate}
\end{theorem}

\subsection{Procedure to Locate a Nearest Definite Pair}\label{sec:procedure_nearest_definite}
Algorithm \ref{alg_dist} below computes the distance $d_{\delta}(A,B)$ defined as in (\ref{eq:distance_definiteness}) 
for a given real number $\delta>0$ by exploiting Theorem \ref{solving_distance}. It also computes 
optimal perturbations $\Delta A,\Delta B$, and determines an angle $\psi$ such that $(\widetilde{A},\widetilde{B})$ 
is a definite pair with positive definite $\widetilde{B}$ and $\lambda_{\min}(\widetilde{B}) = \max \{ \delta , \gamma(A,B) \}$, 
where $\widetilde{A}+\ri \widetilde{B}=e^{-\ri \psi}(A+\Delta A + \ri (B+\Delta B))$.

 \begin{algorithm}
 \begin{algorithmic}[1]
 	\REQUIRE{Hermitian matrices $A,B\in\C^{n\times n}$, a parameter $\delta>0$}
 	\ENSURE{$d_\delta(A,B)$, matrices $\Delta A, \Delta B$ s.t. $\gamma(A+\Delta A, B+\Delta B) = \max \{ \delta,$ $\gamma(A,B) \}$
	and $\| [\Delta A \;\; \Delta B] \|_2 = d_\delta(A,B)$, as well as $\psi \in [0, 2\pi)$ s.t.
	$(\widetilde{A},\widetilde{B}) = e^{-\ri \psi}(A+\Delta A,  B+\Delta B)$ is a definite pair with 
	$\lambda_{\min}(\widetilde{B}) = \max \{ \delta , \gamma(A,B) \}$.}
 	\STATE{$\theta_*\gets\arg\min_{\theta \in [0,2\pi)}\lambda_{\max}(A\cos\theta+B\sin\theta)$} \label{eigopt_problem}
 	\STATE{$H(\theta_*)\gets A\cos\theta_*+B\sin\theta_*$}
 	\STATE{$Q$diag$(\lambda_i)Q^*\gets$ spectral decomposition of $H(\theta_*)$ with $\lambda_n\leq\dots\leq\lambda_1$}
 	\STATE{$d_{\delta}(A,B)\gets\max\{\delta+\lambda_1, 0\}$}
 	\STATE{$\Delta A\gets\cos\theta_\ast Q\text{diag}(\min\{-\delta-\lambda_i,0\})Q^\ast$ \\
 			$\Delta B\gets\sin\theta_\ast Q\text{diag}(\min\{-\delta-\lambda_i,0\})Q^\ast$}
    	\STATE{$\psi\gets \theta_\ast + \pi/2$ is such that $(\widetilde{A},\widetilde{B})$ is a definite pair with positive definite 
    $\widetilde{B}$ and $\lambda_{\min}(\widetilde{B}) = \max \{ \delta , \gamma(A,B) \}$, where 
    $\widetilde{A}+\ri \widetilde{B}=e^{-i\psi}(A+\Delta A + \ri (B+\Delta B))$} \label{calc_angle}
 \end{algorithmic}
 \caption{Distance to a Nearest Definite Pair}
 \label{alg_dist}
\end{algorithm}

We deduce that the matrix $\widetilde{B}$ is positive definite with $\lambda_{\min}(\widetilde{B}) = \delta$
for the particular choice of $\psi$ in line \ref{calc_angle} by employing Theorems \ref{solving_distance}
and \ref{theta_existence}. In particular, by part (iii) of Theorem \ref{solving_distance}, the inner numerical
radius $\zeta(A+\Delta A, B+\Delta B)$ is attained at $\zeta(A+\Delta A, B+\Delta B) e^{\ri(\theta_\ast + \pi)}$.
Hence, it follows from part (iii) of Theorem \ref{theta_existence} that, for the angle 
\[
	\psi	\; := \;	\theta_\ast + \pi - \pi/2		\; = \;		\theta_\ast + \pi/2,
\]
the matrices $\widetilde{A}, \widetilde{B}$ defined by $\widetilde{A} + \ri \widetilde{B} = e^{-i\psi}(A+\Delta A + \ri (B+\Delta B))$
satisfy $\lambda_{\min}(\widetilde{B}) = \gamma(A+\Delta A, B+\Delta B) = \max\{ \delta , \gamma(A,B) \}$.

Note that if $(A,B)$ is definite with $\gamma(A,B) \geq \delta$, then Algorithm \ref{alg_dist} returns
$d_\delta(A,B) = 0$ and $\Delta A = \Delta B = 0$. But $\psi$ in line \ref{calc_angle} in this case
satisfies $\widetilde{A} + \ri \widetilde{B} = e^{-\ri \psi} (A + \ri B)$ with $\lambda_{\min}(\widetilde{B}) = \gamma(A,B)$.

It is argued in \cite{CHENG1999} that the most challenging part of Algorithm \ref{alg_dist} is the solution of 
the optimization problem $\min_{\theta \in [0,2\pi)}\lambda_{\max}(H(\theta))$ globally in line \ref{eigopt_problem}. 
In the next section we present two algorithms to solve these non-convex eigenvalue optimization problems.

\section{Computation of Inner Numerical Radius}\label{small_scale_algorithms}

\subsection{Level-Set Algorithm}\label{sec:level_set_method}
The problem at our hand can be expressed as
\[
	\zeta(C) \; = \; \left| \min_{\theta \in [0,2\pi)} f(\theta) \right|,
	\quad {\rm where} \;\; 
	f(\theta)=\lambda_{\max}(H(\theta)). 
\]
Our first algorithm to minimize $f(\theta)$ globally is analogous to the one to compute the numerical radius described
in \cite{mengi2005}, and is an extension of the Boyd-Balakrishnan algorithm to compute 
the $\mathcal{H}_{\infty}$-norm \cite{Boyd1990}.

Following the practice in \cite{mengi2005}, for a given estimate $\alpha > 0$ for the minimum of $f(\theta)$
over $\theta \in [0, 2\pi)$, the algorithm finds the $\alpha$-level set of $f(\theta)$, that is it determines 
the set of $\theta$ such that $f(\theta) = \alpha$. The open intervals $I_1, \dots, I_k$ satisfying 
\[
	f(\theta) < \alpha \;\;\;\; \forall \theta \in I_\ell
\] 
for $\ell = 1, \dots, k$ are inferred from this $\alpha$-level set. The refined estimate $\tilde{\alpha}$ is set equal 
to the minimum value attained by $f$ over the set of the midpoints of $I_1, \dots, I_k$. The key and challenging 
issue for the algorithm is determining the $\alpha$-level set of $f(\theta)$, and the following result \cite{mengi2005} 
is helpful for this purpose.
\begin{theorem}\label{Theorem_1}
	Let $\alpha > 0$ be a given real number, and
	\[
		R(\alpha)
				= 
		\left[\begin{array}{cc}
			2\alpha I & -C \\
				I &  \; 0
		\end{array} \right],
		\qquad
		S
			=	
		\left[\begin{array}{cc}
			C^\ast & 0 \\
			0 \; & I
		\end{array} \right].
	\]
	The pencil $L(\lambda) = R(\alpha)-\lambda S$ has $e^{i\theta}$ as an eigenvalue for some $\theta \in {\mathbb R}$
	or singular if and only if the Hermitian matrix $H(\theta) = (C e^{-\ri \theta} + C^\ast e^{\ri \theta})/2$ 
	has $\alpha$ as an eigenvalue.
\end{theorem}
We put Theorem \ref{Theorem_1} into use to 
find the $\alpha$-level set of $f(\theta)$ as follows. First we extract the unit generalized eigenvalues 
$e^{\ri \theta'}$ of the pencil $L(\lambda) = R(\alpha)-\lambda S$. The set of such $\theta'$ is a superset 
of the $\alpha$-level set of $f(\theta)$ by Theorem \ref{Theorem_1}. To determine the exact $\alpha$-level set, 
we compute the eigenvalues of $H(\theta')$ for each $\theta'$, and keep the ones for which $H(\theta')$ has 
$\alpha$ as indeed the largest eigenvalue.

Now we are ready to present the algorithm. At the $j$th iteration, given an estimate 
$r^{(j)}$ for $f_\ast := \min_{\theta \in [0,2\pi)}f(\theta)$ such that $r^{(j)} > f_\ast$
and a set of midpoints $\phi_1^{(j)}, \dots, \phi_{m_j}^{(j)}$ of the open intervals where $f(\theta)$ takes values smaller
than $r^{(j)}$, the minimum value that $f(\theta)$ attains over the midpoints $\phi_1^{(j)}, \dots, \phi_{m_j}^{(j)}$
is computed. We set $r^{(j+1)}$ equal to this minimum value. Observe that $r^{(j+1)}$ is a refined estimate, and 
still an upper bound for $f_\ast$. Subsequently, the open intervals $I_1^{(j+1)},\dots,I_{m_{j+1}}^{(j+1)}$ 
where $f(\theta) < r^{(j+1)}$ is determined by exploiting Theorem \ref{Theorem_1} followed 
by the maximum eigenvalue checks. Here we note that $I_{m_{j+1}}^{(j+1)}$ may be of the form 
$I_{m_{j+1}}^{(j+1)}=\left(\ell_{m_{j+1}}^{(j+1)},2\pi\right)$ $\cup\left[0,u_{m_{j+1}}^{(j+1)}\right)$ with $\ell_{m_{j+1}}^{(j+1)}>u_{m_{j+1}}^{(j+1)}$.
Finally, the midpoints $\theta_1^{(j+1)},\dots, \theta_{m_{j+1}}^{(j+1)}$ of $I_1^{(j+1)},\dots,I_{m_{j+1}}^{(j+1)}$ are computed.
Algorithm \ref{Inner_Num_Radius} formally describes this level-set method. 
\begin{algorithm}[t]
	\begin{algorithmic}[1]
		\REQUIRE{A matrix $C\in\C^{n\times n}$.}
		\ENSURE{The sequence $\{ r^{(j)} \}$.}
		
		\STATE $\phi^{(0)}_1  \gets  0$ and $m_0 \gets 1$
		\FOR{$j=0,1,\dots$}
		\STATE{$r^{(j+1)}\gets \min\{f(\phi^{(j)}_k)\; | \; k = 1,\dots, m_j\}$}
		\STATE Determine all of the intervals $I_k^{(j+1)} = \left(\ell_k^{(j+1)},u_k^{(j+1)} \right)$  or
		$I_k^{(j+1)} = \left(\ell_k^{(j+1)},u_k^{(j+1)} + 2\pi \right)$ with $\ell_k^{(j+1)},u_k^{(j+1)} \in [0,2\pi)$
		such that
		\[
			 f(\theta)<r^{(j+1)} \;\;\; \forall \theta\in I_k^{(j+1)} \quad {\rm and}	\quad   f(\ell_k^{(j+1)})=f(u_k^{(j+1)})=r^{(j+1)}
		\]
		for $k=1,\dots,m_{j+1}$.
		\STATE{Form the set $\left\{\phi_1^{(j+1)},\dots,\phi_{m_{j+1}}^{(j+1)}  \right\}$, where $\phi_k^{(j+1)}$ is the midpoint of the open 
		interval $I_k^{(j+1)}$} defined by
		\[
			\phi_k^{(j+1)}
					:=
						\begin{cases}
							\frac{\ell_k^{(j+1)}+u_k^{(j+1)} }{2}, & \text{if } \ell_k^{(j+1)}<u_k^{(j+1)}  \\
							\frac{\ell_k^{(j+1)}+u_k^{(j+1)} +2\pi}{2}\mod 2\pi,  & \text{otherwise}
						\end{cases}.
		\]
		\ENDFOR
	\end{algorithmic}
	\caption{Level-set based algorithm for inner numerical radius}
	\label{Inner_Num_Radius}
\end{algorithm}

The sequence $\left\{r^{(j)}\right\}$ generated by Algorithm \ref{Inner_Num_Radius} is guaranteed to converge
to the minimal value of $f(\theta)$ globally, whose absolute value is equal to $\zeta(C)$. 
This global convergence result can be deduced from the continuity of $f(\theta)$ and the fact that the length 
of the greatest open interval where $f(\theta)$ is smaller than the current estimate is at least halved 
at each iteration \cite[Theorem 4.2]{Boyd1990}. 
\begin{theorem}
	The sequence $\left\{r^{(j)}\right\}$ generated by Algorithm \ref{Inner_Num_Radius} converges to 
	$f_\ast := \min_{\theta \in [0,2\pi)}f(\theta)$.
\end{theorem}

Algorithm \ref{Inner_Num_Radius} converges locally at a quadratic rate under the assumption  
that $\lambda_{\max}(H(\theta))$ is simple at one of its global minimizers. It can be proven for instance 
by following the ideas in the proof for the quadratic convergence result \cite[Theorem 5.1]{Boyd1990} 
regarding the level-set algorithm for ${\mathcal H}_\infty$-norm computation. A formal proof is included in 
Appendix \ref{sec:level-set_quad}.
\begin{theorem}\label{quad_convergence}
	Let $\theta_\ast :=\arg\min_{\theta \in [0,2\pi)}\lambda_{\max}(H(\theta))$, and suppose $\lambda_{\max}(H(\theta))$ 
	is simple at $\theta_\ast$. Then the sequence $\left\{r^{(j)}\right\}$ generated by Algorithm \ref{Inner_Num_Radius} 
	converges to $f_\ast := \min_{\theta \in [0,2\pi)}f(\theta)$ at a quadratic rate.
\end{theorem}

\noindent
\textbf{Remark.} 
In the case that $\lambda_{\max}(H(\theta))$ is not simple at $\theta_*$, we have no longer quadratic convergence. 
This is unlike the analogous algorithms to compute the numerical radius and $\mathcal{H}_{\infty}$-norm,
which converge quadratically regardless of the multiplicity of the eigenvalue at the maximizer of
the largest eigenvalue function involved. The reason is that the largest eigenvalue function at a maximizer
is guaranteed to be at least twice continuously differentiable. On the contrary, this differentiability property 
is not true at a minimizer of the largest eigenvalue function if the eigenvalue is not simple.

\subsection{Support Based Algorithm}\label{sec:support_algorithm}
The algorithm that we employ in this section for computing the inner numerical radius is 
borrowed from \cite{Mengi2014}. Throughout the section, we consider the general setting
of (\ref{eq:main_problem2}). The algorithm that we discuss here to solve (\ref{eq:main_problem2}) 
globally is based on the boundedness of the second derivatives of the objective eigenvalue function. 
It replaces the eigenvalue functions with piece-wise quadratic support functions that underestimate 
$\lambda_{\max}(\mathcal{A}(\omega))$ globally.

We first introduce formally the quadratic support functions, which are the main ingredients of the algorithm
and based on the analytical properties of $\lambda_{\max}(\mathcal{A}(\omega))$. The next result 
states the analytical properties that are relevant to the derivation \cite{Rellich1969, lancaster1964}. 
\begin{lemma}\label{analytical_prop}
   Let $\mathcal{A}(\omega):\R\rightarrow \C^{n\times n}$ be a Hermitian matrix-valued function 
   as in (\ref{eq:main_problem2}). The following hold:
   \begin{itemize}
   	\item [\textbf{(i)}] The eigenvalues $\tilde{\lambda}_1(\omega),\dots,\tilde{\lambda}_n(\omega)$ of 
	$\mathcal{A}(\omega)$ can be permuted in a way so that each of them is a real analytic function of $\omega$.
   	\item [\textbf{(ii)}] For given $\tilde{\omega},p\in\R$, letting 
	$\phi(\alpha):=\lambda_{\max}(\mathcal{A}(\tilde{\omega}+\alpha p))$, 
	the left-hand derivative $\phi_-'(\alpha)$ and the right-hand derivative $\phi_+'(\alpha)$ of $\phi(\alpha)$ 
	exist everywhere, furthermore they satisfy $\phi_+'(\alpha) \geq \phi_-'(\alpha)$ at all $\alpha\in\R$.
   	\item [\textbf{(iii)}] If $\lambda_{\max}(\mathcal{A}(\omega))$ is simple, then it is twice continuously 
	differentiable. At all such $\omega$, we have
   	\[
   		\frac{d\lambda_{\max}(\mathcal{A}(\omega))}{d\omega} \; = \; v^*\frac{d\mathcal{A}(\omega)}{d\omega}v
   	\]
   	and
   	\begin{equation*}
		\begin{split}
   		\frac{d^2\lambda_{\max}(\mathcal{A}(\omega))}{d\omega^2}
			\; = \;
		v^*\frac{d^2\mathcal{A}(\omega)}{d\omega^2}v \; + 	\hskip 30ex \\
		\hskip 20ex
		2\sum_{k=2}^{n}
		\frac{1}{\lambda_{\max}(\mathcal{A}(\omega))-\lambda_k(\mathcal{A}(\omega))}
								\abs{v_k^*\frac{d\mathcal{A}(\omega)}{d\omega}v}^2
		\end{split}
   	\end{equation*}
   	where $\lambda_{k}(\mathcal{A}(\omega))$ denotes the $k$th largest eigenvalue of ${\mathcal A}(\omega)$, and
	$v$, $v_k$ are unit eigenvectors corresponding to 
	$\lambda_{\max}(\mathcal{A}(\omega))$, $\lambda_{k}(\mathcal{A}(\omega))$, respectively,
	for $k=2\dots,n$.
    \end{itemize}
\end{lemma}
In the next result we present the quadratic support functions. 
The proof of the fact that these functions are global under-estimators for the largest eigenvalue functions
follows from part (ii) of Lemma \ref{analytical_prop}. We omit the proof because of its similarity to
the proof of \cite[Theorem 2.2]{Mengi2017}, the analogous result that constructs upper support functions 
for smallest eigenvalue functions.
%
\begin{theorem}[Quadratic Lower Support Functions]
	Suppose $\lambda_{\max}(\mathcal{A}(\omega))$ is simple at $\omega^{(k)}\in\Omega$.
	Additionally, suppose $\gamma$ satisfies $\lambda_{\max}''(\mathcal{A}(\omega))\geq\gamma$ 
	for all $\omega\in\Omega$ such that $\lambda_{\max}(\mathcal{A}(\omega))$ is simple. Then, we have
	\begin{equation}\label{quad_model}
		\lambda_{\max}(\mathcal{A}(\omega)) \;\; \geq \;\; q_k(\omega)	 := 	\lambda_k + 
			\lambda'_k (\omega - \omega^{(k)}) + \frac{\gamma}{2}(\omega - \omega^{(k)})^2
			\;\; \forall \omega \in \Omega
	\end{equation}
	where $\lambda_k := \lambda_{\max} ({\mathcal A}(\omega^{(k)}))$
	$\lambda'_k := \lambda_{\max} ({\mathcal A}(\omega^{(k)})$.
\end{theorem}

We call $q_k(\omega)$ as in (\ref{quad_model}) the quadratic support function about $\omega^{(k)}$.
Such a quadratic support function is defined in terms of a lower bound $\gamma$ for the second
derivatives of the eigenvalue function. This lower bound can occasionally be obtained from the
expression
\begin{equation*}\label{sec_der}
	\frac{d^2\lambda_{\max}(\mathcal{A}(\omega))}{d\omega^2}	=
				v^\ast\frac{d^2\mathcal{A}(\omega)}{d\omega^2}v + 
				2\sum_{k=2}^{n}
					\frac{1}{\lambda_{\max}(\mathcal{A}(\omega))-\lambda_k(\mathcal{A}(\omega))}
					\abs{v_k^*\frac{d\mathcal{A}(\omega)}{d\omega}v}^2.
\end{equation*}
Since the summation term on the right-hand side is non-negative, we must have
\[
	\frac{d^2\lambda_{\max}(\mathcal{A}(\omega))}{d\omega^2}	
			\; \geq \; 
	v^\ast \frac{d^2\mathcal{A}(\omega)}{d\omega^2}v
			\; \geq \;
	-\norm{\frac{d^2\mathcal{A}(\omega)}{d\omega^2}}_2.
\]
Hence $\gamma= - \max_{\omega \in \Omega} \norm{d^2\mathcal{A}(\omega)/d\omega^2}_2$ is a theoretically 
sound choice. 

In the case of computing the inner numerical radius of $A+\ri B$ for a given Hermitian pair $(A,B)$, 
we minimize $\lambda_{\max}(H(\theta))$ where $H(\theta) = A \cos \theta + B \sin \theta$
over $\theta \in [0,2\pi]$. Since
\begin{equation*}
	\begin{split}
	\left|
		\frac{d^2\lambda_{\max}(H(\theta))}{d\theta^2}
	\right|
			\;  & =  \;
	\left|
		v^\ast (A \cos \theta + B \sin \theta) v
	\right|	\\
			\; & \geq \;
	-\norm{A \cos\theta + B \sin\theta}_2
			\; \geq \;
		- \norm{A}_2 - \norm{B}_2
	\end{split}
\end{equation*}
with $v$ representing a unit eigenvector corresponding to the largest eigenvalue of $H(\theta)$,
we can set $\gamma=-\norm{A}_2-\norm{B}_2$.

Finally we present the algorithm based on these support functions. For a given point $\omega^{(0)}$, 
initially the quadratic support function $q_0(\omega)$ about $\omega^{(0)}$ is constructed. Then the 
algorithm generates a sequence $\left\{\omega^{(k)}\right\}$ consisting of estimates for a global minimizer of 
$\lambda_{\max}(\mathcal{A}(\omega))$, a sequence $\left\{ \ell^{(k)} \right\}$ consisting of lower bounds
for the globally smallest value of $\lambda_{\max}({\mathcal A}(\omega))$ and a sequence 
$\left\{\bar{q}_k(\omega)\right\}$ consisting of piece-wise 
quadratic model functions for $\lambda_{\max}({\mathcal A}(\omega))$. At the $k$th iteration, 
the point $\omega^{(k+1)}$ is set equal to a global minimizer of the piece-wise quadratic model function
$\overline{q}_k(\omega) := \max_{j=0,\dots,k}q_j(\omega)$. This is followed by the construction of the quadratic 
support function $q_{k+1}(\omega)$ about $\omega^{(k+1)}$, and inclusion of $q_{k+1}(\omega)$
in the piece-wise quadratic model function. A formal description is given in Algorithm \ref{eigopt}. 
\begin{algorithm}
	\begin{algorithmic}[1]
		\REQUIRE{The matrix-valued function ${\mathcal A}(\omega)$, a closed interval $\Omega\subset\R$, 
		the lower bound $\gamma$ on $\lambda_{\max}''(\mathcal{A}(\omega))$ for all $\omega \in \Omega$
		such that $\lambda_{\max}({\mathcal A}(\omega))$ is simple. }
		\ENSURE{The sequences $\left\{\omega^{(k)}\right\}$ and $\{ \ell^{(k)} \}$.}
		\STATE{$\omega^{(0)}\gets$ an initial point in $\Omega$}
		\STATE{$q_0(\omega)\gets$ the initial quadratic support function about $\omega^{(0)}$}
		\FOR{$k=0,1,\dots$}
		\STATE{$\bar{q}_k(\omega)\gets\max\left\{q_j(\omega)\;|\; j=0,\dots,k\right\}$}
		\STATE{$\omega^{(k+1)}\gets\arg\min_{\omega\in\Omega}\bar{q}_k(\omega)\quad$ and
										$\quad \ell^{(k+1)} \gets \bar{q}_k(\omega^{(k+1)})$}
		\STATE{$q_{k+1}(\omega)\gets$ the quadratic support function about $\omega^{(k+1)}$}
		\ENDFOR
	\end{algorithmic}
	\caption{Support Based Algorithm}
	\label{eigopt}
\end{algorithm}

The next result draws global convergence conclusions regarding the sequences $\{ \omega^{(k)} \}$ and $\{ \ell^{(k)} \}$ 
by Algorithm \ref{eigopt}. This has been proven in \cite[Theorem 8.1]{Mengi2014}.
\begin{theorem}\label{thm:eigopt_global_convergence}
Every convergent subsequence of the sequence $\{ \omega^{(k)} \}$ by Algorithm \ref{eigopt} converges
to a global minimizer of $\lambda_{\max}({\mathcal A}(\omega))$ over $\omega \in \Omega$. Furthermore,
\[
	\lim_{k\rightarrow \infty} \ell^{(k)} \;\; = \;\; \min_{\omega \in \Omega} \: \lambda_{\max} ({\mathcal A}(\omega)).
\]
\end{theorem}
\noindent
If the global minimizer $\omega_\ast$ of $\lambda_{\max}({\mathcal A}(\omega))$ over $\omega \in \Omega$
is unique, then the convergence of the sequence $\{ \omega^{(k)} \}$ itself to $\omega_\ast$ can be asserted.
This is formally presented next.
\begin{theorem}\label{convergence_omega}
	Suppose that $\lambda_{\max}(\mathcal{A}(\omega))$ has a unique global minimizer over all 
	$\omega \in \Omega$, say at  $\omega_\ast$. Then the sequence $\left\{\omega^{(k)}\right\}$ 
	by Algorithm \ref{eigopt} converges to $\omega_\ast$.
\end{theorem}
\begin{proof}
	Assume for a contradiction that $\left\{\omega^{(k)}\right\}$ does not converge to $\omega_*$. Equivalently 
	there exists $\epsilon_0>0$ such that for all $K\in\N$ and for some $k\in\N$ such that $k>K$, 
	we have $\abs{\omega^{(k)}-\omega_*}\geq\epsilon_0$. Hence, we can construct a subsequence 
	$\left\{\omega^{(\ell_{k})}\right\}$ of $\left\{\omega^{(k)}\right\}$ such that for all $k$ the inequality
	\begin{equation}\label{ineq_convergence}
		\abs{\omega^{(\ell_{k})}-\omega_*}\geq\epsilon_0
	\end{equation}
	is satisfied. But this subsequence $\left\{\omega^{(\ell_{k})}\right\}$ is bounded, so it has a
	subsequence $\left\{\omega^{(n_{\ell_{k}})}\right\}$ that is convergent. 
	Now the uniqueness of $\omega_\ast$ combined with Theorem \ref{thm:eigopt_global_convergence}
	leads us to $\lim_{k\rightarrow \infty} \omega^{(n_{\ell_k})} = \omega_\ast$, which contradicts
	(\ref{ineq_convergence}).
\end{proof}

\subsection{Rate-of-Convergence of the Support Based Algorithm}\label{sec:support_rate_convergence}
In numerous numerical experiments, we observe that Algorithm \ref{eigopt} converges at a linear rate 
in the smooth case if $\lambda_{\max}({\mathcal A}(\omega))$ is simple at its global minimizers, but 
a formal proof of this observation is open at the moment.

Here we turn out attention to the non-smooth case, in particular provide a formal rate-of-convergence 
analysis in this case. Remarkably the presence of non-smoothness
accelerates Algorithm \ref{eigopt}. This is quite a contrast to the level-set method of 
Section \ref{sec:level_set_method}, whose quadratic convergence in the smooth case is hindered and limited 
to a linear convergence by the existence of non-smoothness. 

Throughout the rest we assume $\lambda_{\max}({\mathcal A}(\omega))$ has a unique global minimizer, 
say at $\omega_\ast$, and that $\lambda_{\max}({\mathcal A}(\omega_\ast))$ is multiple with
$0 \in {\rm Int} \: \partial \lambda_{\max}({\mathcal A}(\omega_\ast))$ where 
$\partial \lambda_{\max}({\mathcal A}(\omega_\ast))$ is the generalized Clarke
derivative defined as in (\ref{eq:gen_derivative}). (Recall that the condition 
$0 \in {\rm Int} \: \partial \lambda_{\max}({\mathcal A}(\omega_\ast))$ holds generically.)
The eigenvalue function $\lambda_{\max}(\mathcal{A}(\omega))$ is continuous 
and piecewise real analytic at $\omega_\ast$. Indeed there has to be two
real analytic eigenvalue functions $\widetilde{\lambda}_1({\mathcal A}(\omega))$, $\widetilde{\lambda}_2({\mathcal A}(\omega))$ 
such that 
$\widetilde{\lambda}_1({\mathcal A}(\omega_\ast)) = \widetilde{\lambda}_2({\mathcal A}(\omega_\ast))$
and 
$\lambda_{\max}({\mathcal A}(\omega))$ $= 
	\max \{ \widetilde{\lambda}_1({\mathcal A}(\omega)),  \widetilde{\lambda}_2({\mathcal A}(\omega)) \}$
in a neighborhood of $\omega_\ast$ \cite{Rellich1969}. Moreover, without loss of generality, we can assume there 
exists an open interval ${\mathcal I}$ containing $\omega_\ast$ such that
\begin{equation*}
	\lambda_{\max}(\mathcal{A}(\omega))
		\;	=	\;
	\left\{
	\begin{split}
		\widetilde{\lambda}_1(\mathcal{A}(\omega))	&	\quad\quad \omega \in {\mathcal I}, \: \omega \geq \omega_\ast	\\
		\widetilde{\lambda}_2(\mathcal{A}(\omega))	&	\quad\quad \omega \in {\mathcal I}, \; \omega < \omega_\ast
	\end{split}
	\right.
\end{equation*}
and 
$\lambda'_{\ast,+} := \widetilde{\lambda}_1'(\mathcal{A}(\omega_\ast))>0$, 
$\lambda'_{\ast,-} := \widetilde{\lambda}_2'(\mathcal{A}(\omega_\ast))<0$.
Note that $\lambda'_{\ast,+}$ and $\lambda'_{\ast,-}$ correspond to the right-hand and left-hand derivatives, respectively,
of $\lambda_{\max}({\mathcal A}(\omega))$ at $\omega_\ast$. In what follows, we also use the notations
\begin{equation*}
	\begin{split}
	\lambda''_{\ast,+} := \widetilde{\lambda}_1''(\mathcal{A}(\omega_\ast)),	\;\;
	\lambda'''_{\ast,+} := \widetilde{\lambda}_1'''(\mathcal{A}(\omega_\ast)),		\\
	\lambda''_{\ast,-} := \widetilde{\lambda}_2''(\mathcal{A}(\omega_\ast)),	\;\;
	\lambda'''_{\ast,-} := \widetilde{\lambda}_2'''(\mathcal{A}(\omega_\ast)),
	\end{split}
\end{equation*}
which correspond to higher order one-sided derivatives of $\lambda_{\max}({\mathcal A}(\omega))$
at $\omega_\ast$, as well as $\lambda_\ast := \lambda_{\max}({\mathcal A}(\omega_\ast))$.
Additionally, the short-hands $\lambda_k, \lambda'_k, \lambda^{''}_k$ represent 
$\lambda_{\max}({\mathcal A}(\omega^{(k)}))$, $\lambda'_{\max}({\mathcal A}(\omega^{(k)}))$,
$\lambda^{''}_{\max}({\mathcal A}(\omega^{(k)}))$ at an iterate $\omega^{(k)}$ 
of Algorithm \ref{eigopt}. Finally, the lower bound $\gamma$ for the
second derivatives of $\lambda_{\max}({\mathcal A}(\omega))$ is assumed to be negative 
throughout this section without loss of generality.

The following result characterizes the global minimizer $\omega^{(k+1)}$ of $\overline{q}_k(\omega)$
for large $k$. This point always turns out to be the intersection point of two quadratic support functions
about the iterates that are closest to $\omega_\ast$ among the iterates on the left-hand and on the right-hand
side of $\omega_\ast$. Moreover if the distance between the iterates about which these two support functions
are constructed is $h$, then $\omega^{(k+1)}$ is located at a distance of $\Theta(h^2)$ to $\omega_\ast$.
%
\begin{lemma}[Minimizers of the Support Functions]\label{thm:eigopt_support_solutions}
Suppose $\: \omega_\ast$ is the unique global minimizer of $\lambda_{\max}({\mathcal A}(\omega))$
over $\omega \in \Omega$, and the eigenvalue $\lambda_{\max}({\mathcal A}(\omega_\ast))$ is multiple,
$0 \in {\rm Int} \: \partial \lambda_{\max}({\mathcal A}(\omega_\ast))$. 
The sequence $\{ \omega^{(k)} \}$ by Algorithm \ref{eigopt} satisfies the following for all large $k$: 
\begin{enumerate}
	\item[\bf (i)] The point $\; \omega^{(k+1)} \;$ is the intersection point of $q_{\ell(k)}(\omega)$ and $q_{r(k)}(\omega)$,
	where $\ell(k), r(k) \in \{ 0,\dots, k \}$ are given by
	\begin{equation*}
		\begin{split}
			\ell(k)		& \; :=	\;	\arg \min\{ \omega_\ast - \omega^{(j)} \; | \; j \in \{ 0,\dots, k \}
													\;\; {\rm s.t.} \;\; \omega_\ast > \omega^{(j)}	\}
					\quad	{\rm and}		\\
			r(k)		& \; :=	\;	\arg \min\{ \omega^{(j)} - \omega_\ast \; | \; j \in \{ 0,\dots, k \}
													\;\; {\rm s.t.} \;\; \omega_\ast < \omega^{(j)}	\}.
		\end{split}
	\end{equation*}
	\item[\bf (ii)] Letting $h := \max \{ \omega_{\ast} - \omega^{(\ell(k))} , \omega^{(r(k))} - \omega_\ast \}$, we have
	\begin{equation}\label{eq:char_wk}
		\omega^{(k+1)}
			\;	=	\;
		\alpha \cdot \omega_\ast	+	\beta \cdot \left( \frac{\omega^{(\ell(k))}  +  \omega^{(r(k))}}{2} \right)
					+	O(h^3),
	\end{equation}	
	where $\alpha, \beta \in {\mathbb R}^+$ are such that $\alpha + \beta = 1$ and $\beta = \Theta(h)$.
	\item[\bf (iii)] Furthermore, $ | \omega^{(k+1)} - \omega_\ast |  \; = \; \Theta(h^2)$. 
\end{enumerate}
\end{lemma}
\begin{proof}
\textbf{(i)} The real analyticity of $\widetilde{\lambda}_1({\mathcal A}(\omega))$, 
$\widetilde{\lambda}_2({\mathcal A}(\omega))$ imply that these eigenvalue functions are continuously
differentiable. Observe that $\lambda_{\max}'({\mathcal A}(\omega))$ $=$ $\widetilde{\lambda}_1'({\mathcal A}(\omega))$
for all $\omega \in {\mathcal I}$ such that $\omega > \omega_\ast$, and 
$\lambda_{\max}'({\mathcal A}(\omega)) = \widetilde{\lambda}_2'({\mathcal A}(\omega))$
for all $\omega \in {\mathcal I}$ such that $\omega < \omega_\ast$, so there exists an interval
$\widetilde{\mathcal I} := (\omega_\ast - \delta, \omega_\ast + \delta) \subseteq {\mathcal I}$ 
for some $\delta > 0$ such that
\begin{equation*}
	\begin{split}
		\lambda_{\max}'(\mathcal{A}(\omega)) \;\; \geq \;\; \frac{\lambda'_{\ast,+}}{2}	
							&	\quad\quad \forall \omega \in  (\omega_\ast, \omega_\ast + \delta),	\\
		\lambda_{\max}'(\mathcal{A}(\omega)) \;\; \leq \;\; \frac{\lambda'_{\ast,-}}{2}	
							&	\quad\quad \forall \omega \in  (\omega_\ast - \delta, \omega_\ast).
	\end{split}
\end{equation*}
We can choose $\delta$ as small as we wish. In particular, in the subsequent arguments, we assume 
	$\delta \: \leq \: \min \{  (\lambda'_{\ast,-}) / (4\gamma), -(\lambda'_{\ast,+}) / (4\gamma) \}$ 
without loss of generality.

Theorem \ref{convergence_omega} shows that $\omega^{(k)} \rightarrow \omega_\ast$ as $k \rightarrow \infty$,
so for all $k$ large enough $\omega^{(k)}, \omega^{(r(k))}, \omega^{(\ell(k))} \in \widetilde{\mathcal I}$. For
such a large $k$, we have 
\[
	q_{\ell(k)}'(\omega^{(\ell(k))}) \; = \; \lambda_{\max}'({\mathcal A}(\omega^{(\ell(k))})) \; \leq \;  \frac{\lambda'_{\ast,-}}{2}.
\]
Additionally, the inequality $\delta \: \leq \:  (\lambda'_{\ast,-}) / (4\gamma)$ implies
\[
	q_{\ell(k)}'(\omega) \; = \;  \lambda_{\max}'({\mathcal A}(\omega^{(\ell(k))})) + \gamma (\omega - \omega^{(\ell(k))})
				\;\; \leq \;\; \frac{\lambda'_{\ast,-}}{2} - \gamma \delta
				\;\; \leq \;\; \frac{\lambda'_{\ast,-}}{4}.
\]
for all $\omega \in  (\omega_\ast - \delta, \omega^{(\ell(k))})$. This means $q_{\ell(k)}(\omega)$ is
decreasing on this interval, so
\[
	\overline{q}_{k}(\omega^{(\ell(k))}) \;\; = \;\; q_{\ell(k)}(\omega^{(\ell(k))})	
							\;\; < \;\;  q_{\ell(k)}(\omega)
							\;\; \leq \;\; \overline{q}_{k}(\omega)
	\quad	\forall \omega \in  (\omega_\ast - \delta, \omega^{(\ell(k))}).
\]
Consequently, $\omega^{(k+1)}$, the global minimizer of $\overline{q}_{k}(\omega)$, 
cannot lie in $(\omega_\ast - \delta, \omega^{(\ell(k))})$.

Analogous arguments apply to $q_{r(k)}(\omega)$ for $\omega \in \widetilde{\mathcal I}$ such that 
$\omega > \omega^{(r(k))}$. In particular,
\[
	q_{r(k)}'(\omega^{(r(k))}) \; = \; \lambda_{\max}'({\mathcal A}(\omega^{(r(k))})) \; \geq \;  \frac{\lambda'_{\ast,+}}{2},
\]
and, employing $\delta \: \leq \:  -(\lambda'_{\ast,+}) / (4\gamma)$, we have
\[
	q_{r(k)}'(\omega) \; = \;  \lambda_{\max}'({\mathcal A}(\omega^{(r(k))})) + \gamma (\omega - \omega^{(r(k))})
				\;\; \geq \;\; \frac{\lambda'_{\ast,+}}{2} + \gamma \delta
				\;\; \geq \;\; \frac{\lambda'_{\ast,+}}{4}
\]
for all $\omega \in (\omega^{(r(k))}, \omega_\ast + \delta)$. This leads us to the conclusion
\[
	\overline{q}_k(\omega^{(r(k))}) \;\; = \;\; q_{r(k)}(\omega^{(r(k))})	
							\;\; < \;\;  q_{r(k)}(\omega)
							\;\; \leq \;\; \overline{q}_k(\omega)
	\quad
	\forall \omega \in (\omega^{(r(k))}, \omega_\ast + \delta).
\]
Hence, $\omega^{(k+1)} \notin (\omega^{(r(k))}, \omega_\ast + \delta)$ as well.

As $\omega^{(k+1)} \in \widetilde{\mathcal I} = (\omega_\ast - \delta, \omega_\ast + \delta)$, 
the only possibility left out is that $\omega^{(k+1)} \in [\omega^{(\ell(k))}, \omega^{(r(k))}]$. But the 
global minimizer of $\overline{q}_k(\omega)$ in $[\omega^{(\ell(k))}, \omega^{(r(k))}]$ is the intersection
point of $q_{\ell(k)}(\omega)$ and $q_{r(k)}(\omega)$, as the former of these quadratic
functions is decreasing, while the latter is increasing on this interval.

\textbf{(ii)} It follows from part (i) that $\omega^{(k+1)}$ satisfies
$q_{\ell(k)}(\omega^{(k+1)}) = q_{r(k)}(\omega^{(k+1)})$. Now solving this equation for
$\omega^{(k+1)}$ yields
\[
	\omega^{(k+1)}
			=	
	\frac{  \lambda_{\ell(k)}	-	\lambda_{r(k)} 
								+
		 \lambda'_{r(k)} \omega^{(r(k))} 	-	\lambda'_{\ell(k)} \omega^{(\ell(k))}	
								-
		\frac{\gamma}{2} \left( \left[ \omega^{(r(k))} \right]^2	-	\left[ \omega^{(\ell(k))} \right]^2	\right)	}
		{ \left( \lambda'_{r(k)}	-	\lambda'_{\ell(k)} \right)
								-
			\gamma \left( \omega^{(r(k))}	-	\omega^{(\ell(k))}	\right)	}.
\]
By applications of Taylor's theorem with third order remainders to $\lambda_{r(k)}$, $\lambda_{\ell(k)}$, $\lambda'_{r(k)}, \lambda'_{\ell(k)}$,
specifically by expressing them in terms of $\lambda_{\max}({\mathcal A}(\omega))$ and one-sided 
derivatives of $\lambda_{\max}({\mathcal A}(\omega))$ at $\omega_\ast$, we obtain
\begin{equation*}
	\begin{split}
	\omega^{(k+1)}
		\;\;	=	\;\;
	\bigg\{ \;
		\bigg[ \lambda_\ast	+	\lambda'_{\ast,-} ( \omega^{(\ell(k))}  -  \omega_\ast  )
							+		\lambda''_{\ast,-} ( \omega^{(\ell(k))}  -  \omega_\ast  )^2	+ O(h^3) \bigg]
								\;\; -		\hskip 4ex	\\
		\bigg[ \lambda_\ast	+	\lambda'_{\ast,+} ( \omega^{(r(k))}  -  \omega_\ast  )
				+		\lambda''_{\ast,+} ( \omega^{(r(k))}  -  \omega_\ast  )^2	+ O(h^3) \bigg]		\;\; +	  
				\hskip 10ex  \\
		\bigg[ \lambda'_{\ast,+}	+	\lambda''_{\ast,+} ( \omega^{(r(k))}  -  \omega_\ast  )
							+		\lambda'''_{\ast,+} ( \omega^{(r(k))}  -  \omega_\ast  )^2	+ O(h^3) \bigg]	
				\times  \hskip 14ex	\\	
				\bigg[  ( \omega^{(r(k))} - \omega_\ast )  + \omega_\ast \bigg]	\;\; -	\hskip 8ex \\
		\bigg[ \lambda'_{\ast,-}	+	\lambda''_{\ast,-} ( \omega^{(\ell(k))}  -  \omega_\ast  )
							+		\lambda'''_{\ast,-} ( \omega^{(\ell(k))}  -  \omega_\ast  )^2	+ O(h^3) \bigg] \times	\hskip 14ex \\
				\bigg[	( \omega^{(\ell(k))} - \omega_\ast ) + \omega_\ast  \bigg]	\;\;	-	\hskip 8ex \\
					(\gamma/2) \bigg[ \left[ \omega^{(r(k))} \right]^2	-	\left[ \omega^{(\ell(k))} \right]^2	\bigg]	\; \bigg\}	
						\;\;\; \bigg/	  \hskip 34ex \\
	\bigg\{ \;
			\bigg[ \lambda'_{\ast,+}	+	\lambda''_{\ast,+} ( \omega^{(r(k))}  -  \omega_\ast  )
		+		\lambda'''_{\ast,+} ( \omega^{(r(k))}  -  \omega_\ast  )^2	+ O(h^3) \bigg]		\;\; -  \hskip 5ex \\
		\bigg[ \lambda'_{\ast,-}	+	\lambda''_{\ast,-} ( \omega^{(\ell(k))}  -  \omega_\ast  )	
		+		\lambda'''_{\ast,-} ( \omega^{(\ell(k))}  -  \omega_\ast  )^2	+ O(h^3) \bigg]	\;\;	-  \hskip 12ex \\
								\gamma 	\bigg[ \omega^{(r(k))}	-	\omega^{(\ell(k))}	\bigg]		\; \bigg\}.	\hskip 37ex
	\end{split}
\end{equation*}
This can be rearranged into
\[
		\omega^{(k+1)}
			\;	=	\;
		\frac{
		(\alpha_+ - \alpha_-) \cdot \omega_\ast	+	\eta \cdot \left( \frac{\omega^{(\ell(k))}  +  \omega^{(r(k))}}{2} \right) }
		{ \alpha_+ - \alpha_-  + 	\eta}
					\; + \;		O(h^3)
\]
where
\begin{equation*}
	\begin{split}
				\alpha_+	& :=	\left[ \lambda'_{\ast,+}	+	\lambda''_{\ast,+} ( \omega^{(r(k))}  -  \omega_\ast  )
							+		\lambda'''_{\ast,+} ( \omega^{(r(k))}  -  \omega_\ast  )^2 \right],	\\
				\alpha_-	& :=	\left[ \lambda'_{\ast,-}	+	\lambda''_{\ast,-} ( \omega^{(\ell(k))}  -  \omega_\ast  )
							+		\lambda'''_{\ast,-} ( \omega^{(\ell(k))}  -  \omega_\ast  )^2 \right],	\\
				\eta	\;\;	& :=	-\gamma \left( \omega^{(r(k))}		-	\omega^{(\ell(k))}  \right).				
	\end{split}
\end{equation*}
The desired result follows from $\alpha_+ - \alpha_- > 0$, $\; \alpha_+ - \alpha_- = \Theta(1)$, and
by setting 
		$\alpha :=  (\alpha_+ - \alpha_-) / (\alpha_+ - \alpha_-  + 	\eta)$,
		$\beta :=   \eta	/ (\alpha_+ - \alpha_-  + 	\eta)$.
		
\textbf{(iii)} This is immediate from equation (\ref{eq:char_wk}) by observing
\begin{equation*}
	\begin{split}
		\left|
			\omega^{(k+1)}  -  \omega_\ast
		\right|
			& \;	=	\;
		\left|
			\alpha \cdot \omega_\ast	+	\beta \cdot \left( \frac{\omega^{(\ell(k))}  +  \omega^{(r(k))}}{2} \right)
			- (\alpha + \beta) \cdot \omega_\ast
		\right|
					+	O(h^3),	\\
			& \;	=	\;
		\left|
			\beta \cdot \left[	  \left( \frac{\omega^{(\ell(k))}  +  \omega^{(r(k))}}{2} \right) - \omega_\ast	\right]
		\right|		+ O(h^3)
				\;	=	\;	\Theta(h^2).
	\end{split}
\end{equation*}
\end{proof}

The main rate-of-convergence result concerning Algorithm \ref{eigopt} in the non-smooth case
is presented next. It asserts that if the maximum of the errors of the last two
iterates is $h$, the error of the next iterate is $O(h^2)$. A similar rapid convergence
conclusion is drawn for the sequence $\{ \ell^{(k)} \}$ as well.

\begin{theorem}[Rate-of-convergence]\label{thm:quad_convergence}
Suppose $\: \lambda_{\max}({\mathcal A}(\omega))$ has a unique global minimizer
over $\Omega$ at $\: \omega_\ast$. Additionally,
suppose that the eigenvalue $\lambda_{\max}({\mathcal A}(\omega_\ast))$ is multiple, and
$0 \in {\rm Int} \; \partial \lambda_{\max}({\mathcal A}(\omega_\ast))$. 
Following assertions hold for the sequence $\{ \omega^{(k)} \}$ generated by Algorithm \ref{eigopt}
for all $k$ large enough:
\begin{enumerate}
	\item[\bf (i)] 	$k \in \{ \ell(k), r(k) \}$.
	\item[\bf (ii)] 
		$	\left|		\omega^{(k+1)}	- \omega_\ast		\right|
						\; = \;
				O( \max \{ 
							\left|		\omega^{(k)}	- \omega_\ast		\right|,
							\left|		\omega^{(k-1)}	- \omega_\ast		\right|
							\}^2	).
		$
	\item[\bf (iii)] If $\gamma$ is large enough in absolute value, then
		\[		
				\lambda_\ast	-	\ell^{(k+1)}		
						\; = \;
				O(	( \lambda_\ast - \ell^{(k-1)} )^2	).
		\]
\end{enumerate}
\end{theorem}
\begin{proof}
\textbf{(i)} If $\omega^{(k)} > \omega_\ast$, then it is apparent from (\ref{eq:char_wk}) that
for large $k$ we must have $\omega^{(k)} \in (\omega_\ast, (\omega^{(\ell(k-1))} + \omega^{(r(k-1))})/2)$ 
implying $\omega^{(k)} - \omega_\ast < \omega^{(r(k-1))} - \omega_\ast$. Hence, in this case, $r(k) = k$.

Similarly, if $\omega^{(k)} < \omega_\ast$, then $\omega^{(k)} \in ( (\omega^{(\ell(k-1))} + \omega^{(r(k-1))})/2), \omega_\ast)$
for large $k$ by (\ref{eq:char_wk}). This in turn implies $\omega_\ast - \omega^{(k)} < \omega_\ast - \omega^{(\ell(k-1))}$, 
so $\ell(k) = k$.

\textbf{(ii)} Let us suppose $k = r(k)$ without loss of generality. (Otherwise, $k = \ell(k)$ by part (i) and
a similar arguments applies.) This means that $\ell(k) \neq k$, so $\ell(k) = \ell(k-1)$. If $\ell(k) = \ell(k-1) = k-1$,
then part (iii) of Lemma \ref{thm:eigopt_support_solutions} implies
\begin{equation*}
	\begin{split}
		| \omega^{(k+1)}	- \omega_\ast	|	
			&  \;\; = \;\;
	\Theta ( \max \{ 
								\omega_\ast  - \omega^{(\ell(k))}	,
								\omega^{(r(k))}	- \omega_\ast	
							\}^2 ) 	\\
			&  \;\; = \;\;
	\Theta ( \max \{ 
								\omega_\ast  - \omega^{(k-1)}	,
								\omega^{(k)}	- \omega_\ast	
							\}^2 ).
	\end{split}
\end{equation*}
Hence, let us suppose $\ell(k-1) \neq k-1$. But then $r(k-1) = k-1$ by part (i).
Furthermore, as 
$\omega^{(r(k))} = \omega^{(k)} \in (\omega_\ast, (\omega^{(\ell(k-1))} + \omega^{(r(k-1))})/2)$,
we must have 
\[
	\omega^{(k-1)} - \omega_\ast \;\; = \;\; \omega^{(r(k-1))} - \omega_\ast \;\; > \;\; \omega_\ast - \omega^{(\ell(k-1))}
														\;\; = \;\; \omega_\ast - \omega^{(\ell(k))}
\]
from which we deduce
\[
	\max \{ |\omega^{(k-1)} - \omega_\ast |, |\omega^{(k)} - \omega_\ast |  \}
		\;\;	\geq	\;\;
	\max \{  \omega_\ast - \omega^{(\ell(k))} ,  \omega^{(r(k))} - \omega_\ast  \}.
\]
Hence, letting $h := \max \{ |\omega^{(k-1)} - \omega_\ast |, |\omega^{(k)} - \omega_\ast |  \}$,
we have $\max \{  \omega_\ast - \omega^{(\ell(k))} ,  \omega^{(r(k))} - \omega_\ast  \} = O(h)$.
It follows from part (iii) of Lemma \ref{thm:eigopt_support_solutions} that
		$| \omega^{(k+1)} - \omega_\ast |	=	O(h^2)$,
completing the proof.

\textbf{(iii)} 
Part (i) of Lemma \ref{thm:eigopt_support_solutions} asserts 
that the point $\omega^{(k-1)}$ is the intersection point of $q_{\ell(k-2)}(\omega)$ and $q_{r(k-2)}(\omega)$.
Without loss of generality, let us assume $\omega^{(r(k-2))} - \omega_\ast > \omega_\ast - \omega^{(\ell(k-2))}$.
We have
\begin{equation*}
	\begin{split}
	\ell^{(k-1)} & \;\; = \;\;  q_{r(k-2)}(\omega^{(k-1)}) \\	
			& \;\; = \;\;	\lambda_{r(k-2)}	+	\lambda_{r(k-2)}' (\omega^{(k-1)} - \omega^{r((k-2)})
													+ \frac{\gamma}{2} (\omega^{(k-1)} - \omega^{(r(k-2))})^2.
	\end{split}
\end{equation*}
Now applications of Taylor's theorem about $\omega_\ast$ yields
\begin{equation*}
	\begin{split}
	\ell^{(k-1)} = 
		\bigg[ \lambda_{\ast} + \lambda_{\ast,+}' (\omega^{(r(k-2))} - \omega_\ast)	 
							+  \frac{\lambda_{\ast,+}''}{2} (\omega^{(r(k-2))} - \omega_\ast)^2 \bigg]\; + \hskip 11ex \\
		\bigg[ \lambda_{\ast,+}' + \lambda_{\ast,+}''(\omega^{(r(k-2))} - \omega_\ast)  \bigg]  (\omega^{(k-1)} - \omega^{(r(k-2))})
							+ \frac{\gamma}{2} (\omega^{(k-1)} - \omega^{(r(k-2))})^2		\\
				+ \;	O( (\omega^{(r(k-2))} - \omega_\ast)^3 ),	\hskip 10ex
	\end{split}
\end{equation*}
where we use $\omega^{(k-1)} - \omega_\ast = \Theta ((\omega^{(r(k-2))} - \omega_\ast)^2 )$ due to part (iii) of
Lemma \ref{thm:eigopt_support_solutions}. Letting $h := \lambda_\ast - \ell^{(k-1)}$, the last equation yields
\begin{equation*}
	\begin{split}
	\;\; h	\; = \; &	\; - \; \lambda_{\ast,+}' ( \omega^{(k-1)}  -  \omega_\ast  ) \; - \; 
						\frac{\lambda_{\ast,+}''}{2} (\omega^{(r(k-2))} - \omega_\ast)^2 \\
			  	& \hskip 5ex \; - \; \lambda_{\ast,+}''(\omega^{(r(k-2))} - \omega_\ast) (\omega^{(k-1)} - \omega^{(r(k-2))})	\\
				&  \hskip 12ex
						\; - \; \frac{\gamma}{2} (\omega^{(k-1)} - \omega^{(r(k-2))})^2	 
													\; + \; O( (\omega^{(r(k-2))} - \omega_\ast)^3)		\\
			\; = \;	 &	\; - \; \lambda_{\ast,+}' ( \omega^{(k-1)} - \omega_\ast  ) \; + \; 
										\frac{\lambda_{\ast,+}''}{2} (\omega^{(r(k-2))} - \omega_\ast)^2 \\
				& \hskip 11.5ex
				\; - \; \frac{\gamma}{2} (\omega^{(r(k-2))} - \omega_\ast)^2 \; + \; O( (\omega^{(r(k-2))} - \omega_\ast)^3).
	\end{split}
\end{equation*}
Assuming $\gamma$ is large enough, the terms on the order of 
$\Theta (\omega^{(k-1)} - \omega_\ast) = \Theta ((\omega^{(r(k-2))} - \omega_\ast)^2 )$
do not cancel out. This means $\omega^{(k-1)} - \omega_\ast = \Theta(h)$ and 
$( \omega^{(r(k-2))} - \omega_\ast )^2 = \Theta(h)$. 

Letting $h_2 := \lambda_\ast - \ell^{(k)} \leq \lambda_\ast - \ell^{(k-1)}$ (recall that the sequence $\{ \ell^{(k)} \}$
is increasing bounded from above by $\lambda_\ast$), and following similar steps, 
we also deduce $\omega^{(k)} - \omega_\ast = \Theta(h_2) = O(h)$.
Now it follows from part (ii) that $| \omega^{(k+1)} - \omega_\ast | = O(h^2)$. 
The point $\omega^{(k+1)}$ is the intersection point of $q_{\ell(k)}(\omega)$ and $q_{r(k)}(\omega)$,
where $\ell(k) = k$ or $r(k) = k$, so
\begin{equation*}
	\begin{split}
	\ell^{(k+1)}	\; = \;	\overline{q}_k( \omega^{(k+1)})		& \; = \;	q_k(\omega^{(k+1)})		\\
			&\; = \;  \lambda_k + \lambda_k'(\omega^{(k+1)} - \omega^{(k)})  + 
								\frac{\gamma}{2} (\omega^{(k+1)} - \omega^{(k)} )^2.
	\end{split}
\end{equation*}
Assume for now $\omega^{(k)} > \omega_\ast$.
Recalling $\omega^{(k)} - \omega_\ast = \Theta(h_2) = O(h)$, and
once again applying Taylor's theorem to $\lambda_k$, $\lambda_k'$ about $\omega_\ast$ give rise to
\begin{equation*}
	\begin{split}
	\ell^{(k+1)} = & \;\;
		\bigg[ \lambda_{\ast} + \lambda_{\ast,+}' (\omega^{(k)} - \omega_\ast)	 
							+  O ( (\omega^{(k)} - \omega_\ast)^2 ) \bigg] \; + \hskip 15ex \\
		&	\hskip 5ex \bigg[ \lambda_{\ast,+}' + O (\omega^{(k)} - \omega_\ast)  \bigg]  (\omega^{(k+1)} - \omega^{(k)})
							\; + \; \frac{\gamma}{2} (\omega^{(k+1)} - \omega^{(k)})^2	\\		
		=	& \;\;	\lambda_\ast  \; + \; \lambda_{\ast,+}' (\omega^{(k+1)} - \omega_\ast) \; + \; O(h^2) 
			\;\;	=	 \;\; \lambda_\ast + O(h^2),
	\end{split}
\end{equation*}
which in turn implies $\lambda_\ast - \ell^{(k+1)} = O(h^2)$ as desired.
If $\omega^{(k)} < \omega_\ast$, all of the equalities above still hold but by applying Taylor's theorem
on the left-hand side of $\omega_\ast$. This results in the same expressions except that 
occurrences of $\lambda_{\ast,+}'$ are replaced by $\lambda_{\ast,-}'$.
\end{proof}

\subsection{Numerical Experiments} \label{sec:num_exper_small}

\noindent
\textbf{Distance to a Nearest Definite Pair.} $\;$
Consider the matrices $A= {\rm diag}(-3:3) \in {\mathbb R}^{7\times 7}$ and $B \in {\mathbb R}^{7\times 7}$ 
defined by $b_{ij}=1/(i+j)$ except $b_{11} = b_{77} = -1$. This is an indefinite Hermitian pair example taken 
from \cite{CHENG1999}. We run both algorithms to compute the inner numerical radius $\zeta(A+{\rm i}B)$ 
and the distance $d_{\delta}(A,B)$ for $\delta=10^{-8}$.  The computed distance by both algorithms is 
$d_{\delta}(A,B)=0.8118872239262$. Figure \ref{figure} illustrates the field of values of $A+{\rm i}B$, 
$A+\Delta A+ {\rm i}(B+\Delta B)$ and $\widetilde{A}+ \ri \widetilde{B}=e^{-i\psi}(A+\Delta A +{\rm i}(B+\Delta B))$, where 
$\Delta A, \Delta B$ and $\psi$ are as in Algorithm \ref{alg_dist}. Note that $\widetilde{B}$ is positive definite 
with $\lambda_{\min}(\widetilde{B})=10^{-8}$. In each plot, the point where the inner numerical radius is attained 
is marked with a diamond, and the eigenvalues of $A+ \ri B$ are marked with circles. In this example, 
$\lambda_{\max}(A\cos\theta + B\sin\theta)$ turns out to be simple at the global minimizer $\theta_\ast$ 
and Table \ref{quad} illustrates the quadratic convergence of Algorithm \ref{Inner_Num_Radius}.
	\begin{figure}[h]
		$\begin{array}{rl}
			\includegraphics[width=6.2cm]{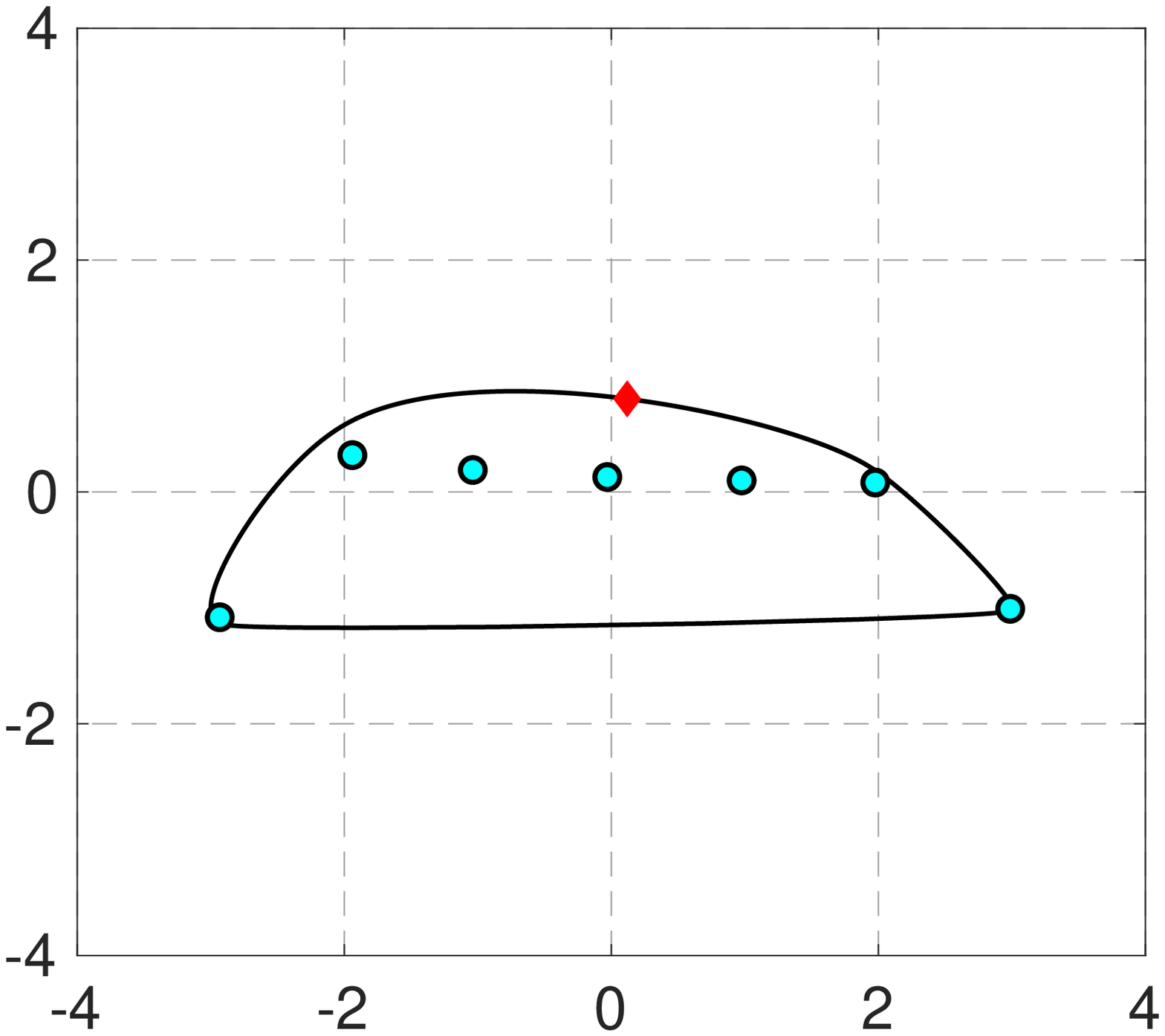} &
			\includegraphics[width=6.2cm]{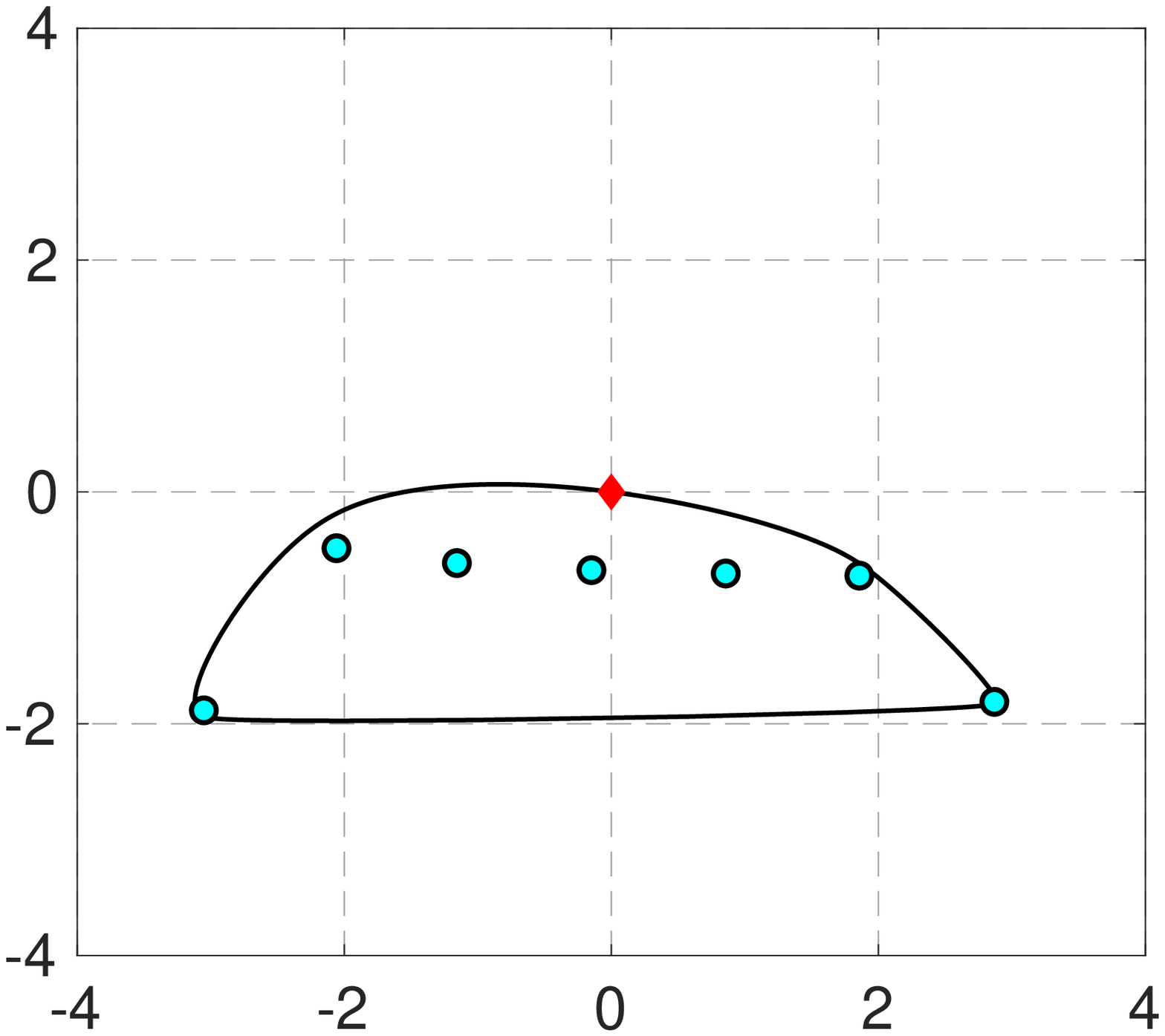}\\
			\multicolumn{2}{c}{\includegraphics[width=6.2cm]{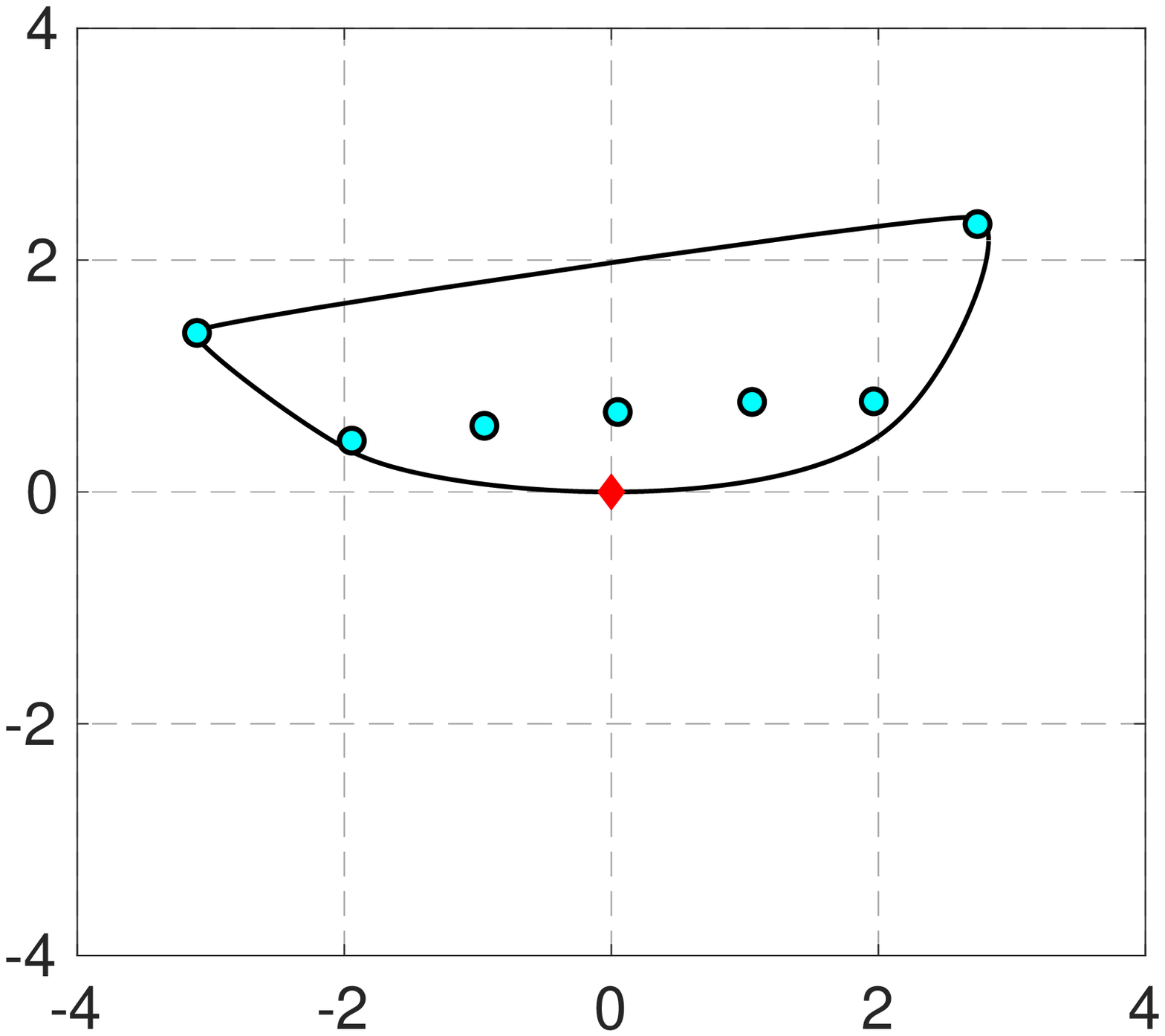}}
		\end{array}$
		\caption{The figure depicts the field of values of $A+ \ri B$ (top left), $(A+\Delta A) + \ri (B+\Delta B)$ (top right), 
		$e^{-\ri \psi}((A+\Delta A) +{\rm i}(B+\Delta B))$ (bottom), where $\Delta A, \Delta B$, $\psi$ are those returned 
		by Algorithm \ref{alg_dist} for the Hermitian pair $(A,B)$ taken from \cite{CHENG1999} and $\delta = 10^{-8}$. 
		In each case, the diamond marks the point where the inner numerical radius is attained, whereas the 
		circles mark the eigenvalues.} 
		\centering
		\label{figure}
	\end{figure}
	
	\begin{table}[tb]
		
		\begin{center}
			\begin{tabular}{|c| c |}
				\hline
				$k$ & $r^{(k)}$ \\
				\hline
				2 & 0.\underline{8}687683091642120 \\
				3 & 0.\underline{811}9559545628993  \\
				4 & 0.\underline{81188722}40421637 \\
				5 & 0.\underline{81188722392623}81 \\
				6 & 0.\underline{8118872239262371} \\
				\hline
			\end{tabular}
		\end{center}
		\caption{The iterates $r^{(k)}$ of Algorithm \ref{Inner_Num_Radius} are listed with respect to $k$.
		These listed values indicate a quadratic rate-of-convergence as expected in theory. }
		\label{quad}
	\end{table}

\medskip

\noindent
\textbf{Testing Hyperbolicity of Quadratic Eigenvalue Problems.}
The quadratic eigenvalue problem (QEP)
\begin{equation}\label{QEP}
Q(\lambda)x=(\lambda^2A+\lambda B +C)x=0
\end{equation}
is said to be hyperbolic if $A,B,C\in\C^{n\times n}$ are Hermitian, $A$ is positive definite, and 
\begin{equation}\label{hyperbolicity}
	(x^*Bx)^2 \; > \; 4(x^*Ax)(x^*Cx) \quad\quad  \forall x\in\C^n \;\; {\rm s.t.} \;\; x \neq 0.
\end{equation}
Let $x$ be an eigenvector and $\lambda$ be the corresponding eigenvalue satisfying (\ref{QEP}).
By multiplying (\ref{QEP}) by $x^\ast$ from left, we obtain the quadratic equation 
$\lambda^2x^\ast Ax  + \lambda x^\ast Bx +  x^\ast Cx  =  0$ with solutions
\begin{equation}\label{hyper_eig_equation}
\frac{-x^\ast B x \pm \sqrt{(x^\ast Bx)^2 - 4(x^\ast A x)(x^\ast Cx)}}{2x^\ast Ax}.
\end{equation}
At least one of these two solutions must be the eigenvalue $\lambda$ corresponding to $x$, which implies, 
assuming the existence of $n$ linearly independent eigenvectors and due to (\ref{hyperbolicity}), 
an hyperbolic QEP has at least $n$ real eigenvalues. More is true; it is shown in 
\cite[Section 7.6]{lancaster1966} that if the QEP is hyperbolic with positive definite $B$ and 
positive semidefinite $C$, then counting also the multiplicities, 
$n$ numbers as in (\ref{hyper_eig_equation}) involving positive square-roots correspond to $n$
eigenvalues of (\ref{QEP}), referred as primary eigenvalues. The remaining $n$ numbers
as in (\ref{hyper_eig_equation}) but involving negative square-roots also correspond to $n$
eigenvalues; these are called the secondary eigenvalues. The eigenvectors associated with the primary 
eigenvalues form a linearly independent set, and the same is true for the secondary eigenvalues.


It is well known \cite{higham2002detecting} that the hyperbolicity of the QEP in (\ref{QEP}) is equivalent to 
the definiteness of $(A_1,B_1)$ with
\begin{equation}\label{hyper_pair}
	A_1
		=
			\left[
				\begin{array}{cc}
					-C & 0 \\
				     \:\: 0 & A
				\end{array}
			\right]
			\qquad
	B_1
		=
			-\left[
				\begin{array}{cc}
					B & A \\
					A & 0
				\end{array}
			\right].
\end{equation}
Consider in particular the QEP with 
\[
	A=I_{4}, 
			\quad 
	B =
		\left[
		  \begin{array}{rrrr}
			8 & -4 & 0& 0 \\
			-4 & 12 &  -4 & 0 \\
			0 & -4 & 12 & -4 \\
			0 & 0 & -4 & 8
		 \end{array}
		\right],
			\quad
	C =
		\left[
		  \begin{array}{rrrr}
			2 & -1 & 0& 0 \\
			-1 & 3 &  -1 & 0 \\
			0 & -1 & 3& -1 \\
			0 & 0 & -1 & 2
		\end{array}
	       \right]
\]
corresponding to a mass, damping, stiffness matrix, respectively.
Applications of Algorithms \ref{Inner_Num_Radius} and \ref{eigopt} for the computation of the inner 
numerical radius of $A_1 + {\rm i} B_1$ yield the global minimizer $\theta_\ast = 2.5682098635$ of
$\lambda_{\max}(A_1 \cos \theta + B_1 \sin \theta)$ with
$\lambda_{\max}(A_1\cos\theta_\ast + B_1\sin\theta_\ast)=$ $-0.4897656697$. 
It follows from Theorem \ref{char_inner} that the field of values of $A_1+ {\rm i} B_1$ does not contain the 
origin, implying that the pair $(A_1, B_1)$ is definite and hence the QEP is hyperbolic.

\medskip

\noindent
\textbf{Comparison of Algorithms \ref{Inner_Num_Radius} and \ref{eigopt}.} Next we compare 
the performances of the level-set based and support based algorithms to compute the inner numerical radius 
of $A_n+ \ri B_n$ for various $n$, where $A_n$ is the Fiedler matrix and $B_n$ is the Moler matrix of size $n\times n$. 
In Table \ref{comparison}, we provide CPU times and the number of iterations required by each algorithm. The reason why the level-set 
approach requires more time is that it computes all eigenvalues of matrices of size $2n\times 2n$, whereas
the support based algorithm computes only the largest eigenvalues of $n \times n$ matrices. 
\begin{table}[tb]
	
	\begin{center}
		\begin{tabular}{ccc||ccc}
			\hline
			\hline
			\multicolumn{3}{c}{{\sc level-set}}			\hskip 10ex								&			\multicolumn{3}{c}{{\sc support-based}}	\hskip 10ex \\
			$n$	&	iter & $t$  &	$n$	& iter & $t$	\\
			\hline
			120		&	5	&	1.60				&	120	&	84	&	0.40	\\
			240		&	5	&	8.90			   	&	240	&	119	&	1.39	  \\
			360		&	6	&	45.44				&	360	&	143	&	4.04		\\
			480		&	6	&	107.99				&	480	&	162	&	6.28	  	\\
			\hline
			\hline
		\end{tabular}
	\end{center}
	\caption{CPU times (in seconds) and the number of iterations required by Algorithms \ref{Inner_Num_Radius} and  \ref{eigopt} 
	to compute the inner numerical radius of $A_n+ \ri B_n$ for various $n$, where $A_n$ is the Fiedler 
	matrix and $B_n$ is the Moler matrix of size $n\times n$.}
	\label{comparison}
\end{table}

\medskip

\noindent
\textbf{The case when $\lambda_{\max}(H(\theta_\ast))$ is not simple.}
The next example illustrates 
the convergence rates of Algorithms \ref{Inner_Num_Radius} and \ref{eigopt} when the largest 
eigenvalue is not simple at the global minimizer. Consider the tridiagonal matrix
\begin{equation}\label{multiple_A}
		\widetilde{A}
			=
				\left[
					\begin{array}{ccccc}
						1		&	{\rm i}	&		&		&		\\
						{\rm i}	&  1		& {\rm i}	&		&		\\
								& {\rm i}	&	a_3	&	\ddots	&		\\
								&		&  \ddots		& \ddots	&	{\rm i}	\\
								&		&		& {\rm i}	&	a_n	\\
					\end{array}
				\right]
			+ 0.5 {\rm i} I_n
			\quad\quad \text{with} \quad a_j =2+\frac{j}{n}
\end{equation}
for $n=10$ and let $A=\widetilde{A}e^{i\pi/6}$. The global minimum of $\lambda_{\max}((A e^{{\rm i}\theta} + A^\ast e^{-{\rm i} \theta})/2)$ is 
attained at $\theta_\ast=3.665191429188092$ and $\lambda_{\max}(H(\theta_\ast))$ has multiplicity $2$. The quadratic convergence 
of Algorithm \ref{Inner_Num_Radius} is no longer true for this example; this is illustrated in Table \ref{multi} 
where the last seven iterates of the algorithm are listed. On the other hand, we observe faster convergence for Algorithm \ref{eigopt} 
consistent with the rate-of-convergence result in Theorem \ref{thm:quad_convergence}. 
This is depicted in Table \ref{table:multi_support}.
	\begin{table}[tb]
		\begin{center}
			\begin{tabular}{|c| c |}
				\hline
				$k$ & $r^{(k)}$ \\
				\hline
				28   &-0.9999999999999492 \\
				29   & -0.9999999999999833 \\
				30 & -0.9999999999999944 \\
				31 & -0.9999999999999982  \\
				32 & -0.9999999999999996 \\
				33 & -0.9999999999999998 \\
				34 & -1.0000000000000000 \\
				\hline
			\end{tabular}
		\end{center}
		\caption{The iterates $r^{(k)}$ of Algorithm \ref{Inner_Num_Radius} to compute the inner numerical radius of the matrix 
			$A$ defined as in (\ref{multiple_A}) are listed.}
		\label{multi}
	\end{table}
	\begin{table}[tb]
		\begin{center}
			\begin{tabular}{ c| c | c }
				\hline
				\hline
				$k$  & $ \ell^{(k+1)}$ & $\abs{\omega^{(k+1)} - \theta_\ast}$ \\
				\hline
				12 &  -1.\underline{0000}24850740653  &0.000903071031368 \\
				13 &  -1.\underline{00000}2750999332  & 0.000041649882203 \\
				14 &  -1.\underline{00000000}1968154 & 0.000001828314169\\
				15 &  -1.\underline{0000000000}11313  &0.000000003906009\\
				16 &  -1.\underline{000000000000000}  & 0.000000000007542\\
				17 &  -1.\underline{000000000000000}  & 0.000000000000000 \\
				\hline
				\hline
			\end{tabular}
		\end{center}
		\caption{The sequence $\{ \ell^{(k+1)} \}$ and the error of the sequence $\{ \omega^{(k+1)} \}$ 
		by Algorithm \ref{eigopt} are listed with respect to $k$ on the example involving the computation of the inner numerical 
		radius of $A=\widetilde{A}e^{i\pi/6}$, where $\widetilde{A}$ is as in (\ref{multiple_A}).}
		\label{table:multi_support}
	\end{table}

\medskip

\noindent
\textbf{Linear systems in saddle point form.} 
This is another example where the eigenvalue function is non-smooth at the optimizer.
The matrix of a saddle point linear system is of the form
\begin{equation}\label{saddle_form}
\mathcal{A}	= \left[
\begin{array}{cc}
A &  \;\;\;\; B^T \\
B & -C
\end{array}
\right],
\end{equation}
where $A \in \R^{n\times n}$ is symmetric positive definite, $B\in\R^{m\times n}$ with $m\leq n$ and $C\in\R^{m\times m}$ is 
symmetric positive semidefinite. The matrix $\mathcal{A}$ is usually large and sparse, and can be reduced to the block diagonal matrix 
${\rm diag}(A,S)$ with $S=-(C+BA^{-1}B^T)$ by row and column operations; indeed $X^T {\mathcal A} X = {\rm diag}(A,S)$
for some invertible $X \in {\mathbb R}^{(n+m)\times (n+m)}$. Now observe that $S$ is symmetric negative semidefinite
implying that $\mathcal{A}$ is indefinite with $n$ positive eigenvalues and ${\rm rank}(S)$ negative eigenvalues.

The indefiniteness of $\mathcal{A}$ is a hurdle for iterative solvers such as Krylov subspace methods; it slows down the
convergence. It has been shown in \cite{liesen2008nonsymmetric} that, letting $\mathcal{J} := {\rm diag}(I_n, -I_m)$,
if there is a real scalar $\mu$ such that $\mathcal{M}(\mu)=\mathcal{A}-\mu\mathcal{J}$ is positive definite, then a conjugate
gradient iteration that depends on this value of $\mu$  can be constructed to solve the linear system 
$\mathcal{J}\mathcal{A}x=\mathcal{J}b$. Clearly the positive definiteness of $\mathcal{M}(\mu)$ for some $\mu\in\R$ 
is equivalent to the positive definiteness of $-\mathcal{A}\sin\theta+\mathcal{J}\cos\theta$ for some $\theta\in[0,2\pi]$. 
Now Theorem \ref{theta_existence} suggests testing the definiteness of the pair $(\mathcal{A},\mathcal{J})$.
If this pair is definite, then, by part (iii) of Theorem \ref{theta_existence}, the matrix $-\mathcal{A}\sin\varphi+\mathcal{J}\cos\varphi$ 
is positive definite for $\varphi := -\pi/2 + \phi$ where $\phi$ is the angle such that $\gamma({\mathcal A},{\mathcal J}) e^{\ri \phi}$
is the point in $F({\mathcal A},{\mathcal J})$ closest to the origin. This in turn implies that $\mathcal{M}(\mu)$ is positive definite 
for $\mu := -\cos\varphi/\sin\varphi$.

We consider a linear system in saddle point form that arises from a stable discretization of a Stokes equation 
with the coefficient matrix $\mathcal{A}$ of the form (\ref{saddle_form}). This linear system is generated by the 
MATLAB package ``Incompressible Flow Iterative Solution Software (IFISS) version 3.5" \cite{silvester2010incompressible}; 
more specifically, the sparse matrices $A,B,C$ are constructed by running the script file \texttt{stokes\_testproblem} 
with the default options, resulting in $A$ of size $n=578$ and $C$ of size $m=256$. For compatibility with the particular examples 
worked through in \cite{liesen2008nonsymmetric,guo2009improved}, we shift $A$ by $0.0764I_n$, 
and run support based algorithm to determine the definiteness of the pair $(\mathcal{A},\mathcal{J})$. We detect that 
the pair is definite. Moreover computations yield $\mu=0.0541$ for which $\mathcal{M}(\mu)$ is positive definite,
indeed $\lambda_{\min}(\mathcal{M}(\mu))=0.0223$. The non-smoothness in these computations is encountered
in a strong fashion, as $\lambda_{\max}(\mathcal{A}\cos\theta+\mathcal{J}\sin\theta)$ at the minimizing
$\theta$ has multiplicity three. Table \ref{Saddle_mult_support} indicates rapid convergence for the
sequences $\{ \ell^{(k)} \}$ and $\{ \omega^{(k)} \}$ by Algorithm \ref{eigopt}. In particular, according to the table, 
the sequences $\{ \ell^{(2k)} \}$, $\{ \ell^{(2k+1)} \}$ appear to be converging at a quadratic rate. Additionally, the 
number of accurate decimal digits of $\omega^{(k)}$ is doubled at every two iterations. Once again, these are
consistent with the assertions of Theorem \ref{thm:quad_convergence}. 

 \begin{table}[tb]
	\begin{center}
		\begin{tabular}{ c| c | c }
			\hline
			\hline
			$k$  & $\ell^{(k+1)}$ &  $\abs{ \omega^{(k+1)} - \theta_\ast}$ \\
			\hline
			11 &  -0.\underline{0222}31442152684 &0.000008162291389 \\
			12 &  -0.\underline{022224901}919276 & 0.000006757114932 \\
			13 &  -0.\underline{022224901}723492& 0.000000000136544 \\
			14 &  -0.\underline{02222490166667}1 & 0.000000000058712 \\
			15 &  -0.\underline{022224901666670}  & 0.000000000000000 \\
			\hline
			\hline
		\end{tabular}
	\end{center}
	\caption{This table is analogous to Table \ref{table:multi_support}, but it concerns
	the example of the saddle point linear system arising from the Stokes equation,
	in particular the positive definiteness of ${\mathcal A} - \mu {\mathcal J}$ for some 
	$\mu$, where ${\mathcal A} \in {\mathbb R}^{(n+m)\times (n+m)}$ is as in (\ref{saddle_form})
	and $\mathcal{J} := {\rm diag}(I_n, -I_m)$.}
	\label{Saddle_mult_support}
 \end{table}

\section{Subspace Framework for Large-Scale Computation of Inner Numerical Radius}\label{sec:subspace}

We now deal with the general univariate eigenvalue optimization problems of the form (\ref{eq:main_problem2}) 
when the Hermitian matrices $A_1, \dots, A_d$ involved are of large size. This setting encompasses 
the eigenvalue optimization characterization (\ref{Main_problem}) for the inner numerical radius when 
the matrices $A, B$ are large. Hence, the approach discussed in this section can be incorporated
into Algorithm \ref{alg_dist} to deal with large Hermitian matrix pairs.
For instance, it can be used to determine the definiteness of a large-scale 
Hermitian pair $(A,B)$, or for such a large definite Hermitian pair, it can be
used to find an angle $\varphi$ such that $\widetilde{A} + \ri \widetilde{B} \; = \; e^{-\ri \varphi} (A + \ri B)$ 
with $\lambda_{\min}(\widetilde{B}) = \gamma(A,B)$.

We rely on the subspace framework from \cite{kangal2018}, which is based on 
the conversion of the original problem into a small-scale one by means of orthogonal projections and 
restrictions to certain subspace. It is established in the literature that the framework converges globally 
and at least at a superlinear rate locally \cite{kangal2018,Kressner2017} under the
assumption that the objective eigenvalue function $\lambda_{\max}({\mathcal A}(\omega))$ is simple at the 
converged global minimizer. Here we generalize the basic subspace procedure in \cite{kangal2018}
slightly, taking into account also the non-smooth case when $\lambda_{\max}({\mathcal A}(\omega))$ is 
multiple at the converged global minimizer. We prove that the convergence of the generalized subspace 
procedure occurs at a quadratic rate asymptotically in the non-smooth case. This rate-of-convergence 
issue in the non-smooth setting has been left open by the previous works.

\subsection{Generalized Subspace Procedure}
The subspace procedure is built around the reduced problems of the form
\begin{equation}\label{reduced_opt}
	\min_{\omega\in\Omega}	\: \lambda_{\max}(\mathcal{A}^{\mathcal{V}}(\omega)),
\end{equation}
where
\begin{equation*}
	\begin{split}
		\mathcal{A}^{\mathcal{V}}(\omega)
			 \; & := \;
		V^\ast {\mathcal A}(\omega) V	\\
			 \; & = \;
		f_1(\omega)V^\ast A_1V + \dots +f_d(\omega)V^\ast A_dV
	\end{split}
\end{equation*}
for a given small dimensional subspace ${\mathcal V}$, say ${\rm dim} \: {\mathcal V} = k$,
and an $n\times k$ matrix $V$ whose columns form an orthonormal basis for ${\mathcal V}$. 
Note that (\ref{reduced_opt}) involves smaller $k\times k$ eigenvalue 
problems compared with (\ref{eq:main_problem2}), which involves $n\times n$ eigenvalue problems.

In the remaining part of this subsection, we shall describe a procedure to construct a small dimensional
subspace ${\mathcal V}$ such that the global minimizers and the globally minimal values of
$\lambda_{\max}({\mathcal A}(\omega))$ and $\lambda_{\max}^{\mathcal V}({\mathcal A}(\omega))$
are nearly the same. To this end, we remind a lemma from \cite{kangal2018} that relates the eigenvalues 
of $\mathcal{A}(\omega)$ and $\mathcal{A}^{\mathcal{V}}(\omega)$. 
Here and throughout the rest of this section $\lambda_j({\mathcal A}(\omega))$ and $v_j({\mathcal A}(\omega))$
denote the $j$th largest eigenvalue and a corresponding unit eigenvector of ${\mathcal A}(\omega)$.
\begin{lemma}\label{monotonicity_Hermite}
	Let $\mathcal{V}_1,\mathcal{V}_2$ be two subspaces of $\; \C^n$ such that 
	$\mathcal{V}_1\subseteq \mathcal{V}_2$. We have the following for each $k = 1, \dots , \dim \: {\mathcal V}_1$:
	\begin{enumerate}
		\item [\textbf{(i)}] (Monotonicity) 
			$\lambda_k(\mathcal{A}^{\mathcal{V}_1}(\omega))
						\leq
			\lambda_k(\mathcal{A}^{\mathcal{V}_2}(\omega))\leq\lambda_k(\mathcal{A}(\omega))
				\;\;	\forall \omega \in \Omega$.
		\item [\textbf{(ii)}] (Hermite Interpolation) 
			For a given $\omega\in\Omega$, if $\; \mathcal{V}_1$ contains 
			$v_1(\mathcal{A}(\omega)),$ $\dots,$ $v_j({\mathcal A}(\omega))$, then 
			the following hold for each $k = 1,\dots,j$:
		\begin{itemize}
			\item $\lambda_k(\mathcal{A}^{\mathcal{V}_1}(\omega))=\lambda_k(\mathcal{A}(\omega))$, 
			\item if $\; \lambda_k(\mathcal{A}(\omega))$ is simple, then so is 
			$\lambda_k(\mathcal{A}^{\mathcal{V}_1}(\omega))$, and satisfies \\
			$\lambda_k'(\mathcal{A}^{\mathcal{V}_1}(\omega))=\lambda_k'(\mathcal{A}(\omega))$.
		\end{itemize}
	\end{enumerate}
\end{lemma}

The subspace procedure is presented formally in Algorithm \ref{subspace_algorithm}.
At every iteration, the subspace procedure first solves a projected small-scale problem for a given subspace 
$\mathcal{V}$. Then, denoting the global minimizer of this small problem with $\omega_\ast$, the subspace is 
expanded with the inclusion of an eigenvector corresponding to $\lambda_{\max}(\mathcal{A}(\omega_\ast))$, as well 
as eigenvectors corresponding to other eigenvalues of ${\mathcal A}(\omega_\ast)$ that are at most $\epsilon$ away 
from $\lambda_{\max}(\mathcal{A}(\omega_\ast))$.  
The following interpolation result between the eigenvalues of the full and projected problems 
generated by Algorithm \ref{subspace_algorithm} is an immediate corollary of part (ii) of 
Lemma \ref{monotonicity_Hermite}.
\begin{theorem}\label{thm:Hermite_interpolate_fw}
The following are satisfied by the sequences $\{ \omega^{(k)} \}$ and $\{ {{\mathcal V}_k} \}$ 
generated by Algorithm \ref{subspace_algorithm} for each $j$, each $k = 1,\dots,j$,
each $p = 1,\dots, \ell_k$:
\begin{enumerate}
	\item[\bf (i)] $\lambda_p ({\mathcal A}^{{\mathcal V}_j}(\omega^{(k)}))  
					=
				 \lambda_p ({\mathcal A}(\omega^{(k)}))$.
	\item[\bf (ii)] if $\; \lambda_p({\mathcal A}(\omega^{(k)}))$ is simple, then the same holds for 
				$\lambda_p({\mathcal A}^{{\mathcal V}_j}(\omega^{(k)}))$, and \\
				$\lambda_p'({\mathcal A}^{{\mathcal V}_j}(\omega^{(k)})) = \lambda_p'({\mathcal A}(\omega^{(k)}))$.
\end{enumerate}
\end{theorem}

\begin{algorithm}[t]
	\begin{algorithmic}[1]
		\REQUIRE{Matrix-valued function $\mathcal{A}(\omega)$, closed interval $\Omega\subset\R$, 
						and $\epsilon \in {\mathbb R}^+$.}
		\ENSURE{The sequence $\{ \omega^{(k)} \}$.}
		\STATE{$\omega^{(1)}\gets$ a random point in $\Omega$}
		\STATE{$v^{(1)}_1 \gets$ eigenvector corresponding to $\lambda_{\max}(\mathcal{A}(\omega^{(1)}))$}
		\label{comp_eigvecs01}
		\STATE{$v_2^{(1)},\dots,v_{\ell_1}^{(1)}\gets$ eigenvectors corresponding to eigenvalues\\
		\hskip 5ex
		$\lambda_2({\mathcal A}(\omega^{(1)}))$, $\dots$, $\lambda_{\ell_1}({\mathcal A}(\omega^{(1)}))$ 
		of ${\mathcal A}(\omega^{(1)})$ such that \\
		\hskip 14ex $\lambda_{\max}(\mathcal{A}(\omega^{(1)})) - \lambda_j(A(\omega^{(1)}) \leq \epsilon \;\;$ for $\;\; j = 2, \dots, \ell_1$}\label{comp_eigvecs02}
		\STATE{$\mathcal{V}_1\gets$span$\left\{v^{(1)}_1,v_2^{(1)},\dots,v_{\ell_1}^{(1)}\right\}$}
		\FOR{$k=1,2,\dots$}
		\STATE{$\omega^{(k+1)}\gets\arg\min_{\omega\in\Omega}\lambda_{\max}(\mathcal{A}^{\mathcal{V}_k}(\omega))$}
		\STATE{$v^{(k+1)}_1 \gets$ eigenvector corresponding to $\lambda_{\max}(\mathcal{A}(\omega^{(k+1)}))$} \label{comp_eigvecs1}
		\STATE{$v_2^{(k+1)},\dots,v_{\ell_{k+1}}^{(k+1)}\gets$ eigenvectors corresponding to eigenvalues\\
		\hskip 4ex  
		$\lambda_2({\mathcal A}(\omega^{(k+1)}))$, $\dots$, $\lambda_{\ell_{k+1}}({\mathcal A}(\omega^{(k+1)}))$ of ${\mathcal A}(\omega^{(k+1)})$ such that \\
		\hskip 8ex 
		$\lambda_{\max}(\mathcal{A}(\omega^{(k+1)})) - \lambda_j(A(\omega^{(k+1)}) \leq \epsilon \;\;$ for $\;\;j = 2, \dots, \ell_{k+1}$}\label{comp_eigvecs2}
		\STATE{$\mathcal{V}_{k+1}\gets\mathcal{V}_k \: \oplus \: {\rm span}\left\{v^{(k+1)}_1,v_2^{(k+1)},\dots,v_{\ell_{k+1}}^{(k+1)}\right\}$}
		\ENDFOR
	\end{algorithmic}
	\caption{The Subspace Procedure}
	\label{subspace_algorithm}
\end{algorithm}

In the case of inner numerical radius, the projected small-scale 
problems can be solved globally and efficiently by means of the algorithms in Section \ref{small_scale_algorithms}. 
The support based algorithm in Section \ref{small_scale_algorithms} is applicable to solve the projected problems 
associated with various other eigenvalue optimization problems  of the form (\ref{eq:main_problem2}), as long 
as a global lower bound $\gamma$ on $\lambda_{\max}({\mathcal A}''(\omega))$ at the points of
differentiability is available.
The main computational burden of the subspace procedure stems from lines 
\ref{comp_eigvecs01}, \ref{comp_eigvecs02}, \ref{comp_eigvecs1}, \ref{comp_eigvecs2}, which require the computation 
of the eigenvectors of the full problem.

The subsequent two subsections are devoted to analyses of the convergence properties of 
Algorithm \ref{subspace_algorithm}. The next subsection provides formal arguments in support 
of the fact that every convergent subsequence of the sequence $\{ \omega^{(k)} \}$ by the 
algorithm converges to a global minimizer of $\lambda_{\max}({\mathcal A}(\omega))$. Then 
Section \ref{sec:subspace_rate-of-convergence} addresses how quickly this convergence occurs.

\subsection{Global Convergence}
For the sake of rigor, here and in the next section, we consider the problem at our hands in the infinite 
dimensional setting. In particular, in this analysis the matrix-valued functions are replaced by self-adjoint 
compact operators $\mathcal{A}(\omega):\ell^2(\N)\rightarrow\ell^2(\N)$, where $\ell^2(\N)$ denotes the 
Hilbert space consisting of square summable infinite sequences of complex numbers equipped with 
the inner product $\langle v,w\rangle=\sum_{k=1}^{\infty}v_k w_k$ and the norm 
$\norm{v}=\sqrt{\sum_{k=1}^{\infty} |v_k|^2}$. The compact self-adjoint operator $\mathcal{A}(\omega)$ 
dependent on the parameter $\omega$ is still assumed to be of the form specified in (\ref{eq:main_problem2}); 
only now $A_j:\ell^2(\N)\rightarrow\ell(\N)$ are self-adjoint compact operators for $j=1,\dots,d$. Intuitively 
$\mathcal{A}(\omega)$ for each $\omega$, as well as $A_1,\dots,A_d$, can be considered as infinite 
dimensional Hermitian matrices. 

The global convergence of the subspace procedure is a consequence of the monotonicity and 
interpolation properties, as well as the uniform Lipschitz continuity of the reduced eigenvalue function 
formally stated below. 
\begin{lemma}[Uniform Lipschitz Continuity \cite{kangal2018}]\label{thm:Lip_cont}
	There exists a positive real number $\eta$ such that for a prescribed $J \in {\mathbb Z}^+$ for each $j = 1,\dots,J$
	and for all subspaces $\mathcal{V}$ of $\ell^2(\N)$ the 
	following holds:
	\begin{eqnarray*}
		\abs{\lambda_j(\mathcal{A}(\omega_1)) - 
				\lambda_j(\mathcal{A}(\omega_2))} 
						\leq 
			\eta	\abs{\omega_1 -\omega_2} \quad \forall \omega_1,\omega_2\in\Omega, \;\; {\rm and} \\
		\abs{\lambda_j(\mathcal{A}^{\mathcal{V}}(\omega_1)) - 
				\lambda_j(\mathcal{A}^{\mathcal{V}}(\omega_2))} 
						\leq 
			\eta	\abs{\omega_1 -\omega_2} \quad \forall \omega_1,\omega_2\in\Omega.
	\end{eqnarray*}
\end{lemma}

Now let $\left\{\omega^{(\ell_{k})}\right\}$ be a convergent subsequence of $\left\{\omega^{(k)}\right\}$.
The interpolation property (part (i) of Theorem \ref{thm:Hermite_interpolate_fw}) implies
	\begin{equation}\label{inter}
	   	\min_{\omega\in\Omega} \: \lambda_{\max}(\mathcal{A}(\omega)) 
				\; \leq \;
		\lambda_{\max}(\mathcal{A}(\omega^{\ell_k}))
				\; = \;
		\lambda_{\max}(\mathcal{A}^{\mathcal{V}_{\ell_k}}(\omega^{\ell_{k}})),
	\end{equation}
while the monotonicity property (part (i) of Lemma \ref{monotonicity_Hermite}) implies
	\begin{equation}\label{mon}
		\begin{split}
		\min_{\omega\in\Omega} \: \lambda_{\max}(\mathcal{A}(\omega))
				&\; \geq \;
		\min_{\omega\in\Omega} \: \lambda_{\max}(\mathcal{A}^{\mathcal{V}_{\ell_{k+1}-1}}(\omega)) \\
				&\; = \;
		\lambda_{\max}(\mathcal{A}^{\mathcal{V}_{\ell_{k+1}-1}}(\omega^{(\ell_{k+1})}))
				\; \geq \;
		\lambda_{\max}(\mathcal{A}^{\mathcal{V}_{\ell_{k}}}(\omega^{(\ell_{k+1})})).
		\end{split}
	\end{equation}
Hence, $\min_{\omega\in\Omega} \: \lambda_{\max}(\mathcal{A}(\omega))$ is squeezed in between
$\lambda_{\max}(\mathcal{A}^{\mathcal{V}_{\ell_k}}(\omega^{\ell_{k}}))$ and 
$\lambda_{\max}(\mathcal{A}^{\mathcal{V}_{\ell_{k}}}(\omega^{(\ell_{k+1})}))$, the gap between
which is decaying to zero as $k \rightarrow \infty$ due to uniform Lipschitz continuity
(Lemma \ref{thm:Lip_cont}). This leads to the following global convergence result. 
The details of the proof are omitted, as the proof is identical to the one for \cite[Theorem 3.1]{kangal2018}.

\begin{theorem}[Global Convergence]\label{global_convergence}
	Every convergent subsequence of the sequence  $\left\{\omega^{(k)}\right\}$ generated by Algorithm \ref*{subspace_algorithm} 
	in the infinite dimensional setting converges to a global minimizer of $\lambda_{\max}(\mathcal{A}(\omega))$ over 
	$\omega\in\Omega$. Moreover,
	\begin{equation}\label{eq:glob_conv_p2}
		\lim_{k\rightarrow\infty}\lambda_{\max}(\mathcal{A}^{\mathcal{V}_k}(\omega^{(k+1)}))
			 =  \lim_{k\rightarrow\infty}\min_{\omega\in\Omega}\: \lambda_{\max}(\mathcal{A}^{\mathcal{V}_k}(\omega))
			 =  \min_{\omega\in\Omega} \: \lambda_{\max}(\mathcal{A}(\omega)).
	\end{equation}
\end{theorem}

\subsection{Rate-of-Convergence}\label{sec:subspace_rate-of-convergence}
In this section, as has been done earlier for the rate-of-convergence analysis of the support based 
algorithm, we again assume $\lambda_{\max}(\mathcal{A}(\omega))$ has a unique global
minimizer over $\Omega$, say at $\omega_\ast$. It turns out that Theorem \ref{convergence_omega}
that is established for the support based algorithm applies to this setting as well;
this is a consequence of Theorem \ref{global_convergence}, in particular its assertion that every convergent 
subsequence of the sequence $\{ \omega^{(k)} \}$ by the subspace framework is guaranteed to 
converge to a global minimizer. Hence, the sequence $\{ \omega^{(k)} \}$ itself converges 
to the unique global minimizer $\omega_\ast = \arg\min_{\omega\in\Omega}\lambda_{\max}(\mathcal{A}(\omega))$.

Here we are concerned with how quickly $\{ \omega^{(k)} \}$ by
Algorithm \ref{subspace_algorithm} converges to $\omega_\ast$. 
If the eigenvalue $\lambda_{\max}(\mathcal{A}(\omega_\ast))$ is simple,
then the eigenvalue function $\lambda_{\max}(\mathcal{A}(\omega))$ is real analytic at $\omega_\ast$.
In this case, it has been shown in \cite{Kressner2017} that the precise R-order of convergence 
of $\{ \omega^{(k)} \}$ to $\omega_\ast$ is $1+\sqrt{2}$, i.e., there exists a sequence
$\{ \varepsilon^{(k)} \}$ converging to zero such that $| \omega^{(k)} - \omega_\ast | \leq \varepsilon^{(k)}$ 
for all $k$ large enough and $ \varepsilon^{(k+1)} = O((\varepsilon^{(k)})^{1 + \sqrt{2}})$.

Otherwise, $\lambda_{\max}(\mathcal{A}(\omega_\ast))$ is multiple with algebraic multiplicity $K \geq 2$.
Throughout this section, it is assumed that 
\begin{equation}\label{eq:assumption_gen_der}
	0 \notin {\rm bd} \: \partial \lambda_j({\mathcal A}(\omega_\ast)),	\quad	{\rm for} \;\; j = 1,\dots,K,
\end{equation}
where $\lambda_j({\mathcal A}(\omega))$ denotes the $j$th largest eigenvalue of ${\mathcal A}(\omega)$.
This assumption is equivalent to having non-zero one-sided derivatives of
$\lambda_{j}({\mathcal A}(\omega))$ at $\omega_\ast$ for $j = 1,\dots, K$,
and holds generically. As already discussed in Section \ref{sec:support_rate_convergence}, 
$\lambda_{\max}(\mathcal{A}(\omega))$ is continuous and piecewise real analytic at $\omega_\ast$,
but usually not differentiable. Furthermore, 
part (i) of Lemma \ref{analytical_prop} asserts the existence of real analytic eigenvalue functions
$\widetilde{\lambda}_1({\mathcal A}(\omega)), \dots, \widetilde{\lambda}_n({\mathcal A}(\omega))$ 
of ${\mathcal A}(\omega)$. Precisely $K$ of these $n$ real analytic functions at $\omega_\ast$
must be equal to $\lambda_{\max}({\mathcal A}(\omega_\ast))$ with nonzero derivatives;
a zero derivative for one of these $K$ functions at $\omega_\ast$ contradicts (\ref{eq:assumption_gen_der}).
Hence, without loss of generality, suppose
\begin{equation*}
		\widetilde{\lambda}_{1}({\mathcal A}(\omega_\ast)) \;\; = \;\; \dots \;\; = \;\;  \widetilde{\lambda}_{K}({\mathcal A}(\omega_\ast)) 	\;\; = \;\; 
		\lambda_{\max}({\mathcal A}(\omega_\ast))
\end{equation*}
are such that
\begin{equation*}
		\widetilde{\lambda}'_{1}({\mathcal A}(\omega_\ast)), \dots ,  
									\widetilde{\lambda}'_{P}({\mathcal A}(\omega_\ast))  >  0,	\quad \quad
		\widetilde{\lambda}'_{P+1}({\mathcal A}(\omega_\ast)), \dots ,  \widetilde{\lambda}'_{K}({\mathcal A}(\omega_\ast))  < 0.
\end{equation*}
Our arguments make use of the gap
\begin{equation}\label{eq:eig_gap}
	\begin{split}
	\varphi \;\; & := \;\; \lambda_{\max}({\mathcal A}(\omega_\ast)) \;\; - \;\;
				\lambda_{K+1} ({\mathcal A}(\omega_\ast)) \\
		\;\; & = \;\;	\lambda_{\max}({\mathcal A}(\omega_\ast)) \;\; - \;\;
				\max \left\{ \widetilde{\lambda}_j ({\mathcal A}(\omega_\ast)) \; \big| 
							\; j = K + 1, \dots, n \right\}
	\end{split}
\end{equation}
as well as the real analytic eigenvalue functions of ${\mathcal A}^{{\mathcal V}_k}(\omega)$  which we denote with 
$\widetilde{\lambda}_1({\mathcal A}^{{\mathcal V}_k}(\omega)), \dots, \widetilde{\lambda}_{d_k}({\mathcal A}^{{\mathcal V}_k}(\omega))$,
where $d_k := \dim \: {\mathcal V}_k$. 

The Hermite interpolation property extends to $\widetilde{\lambda}_j({\mathcal A}(\omega))$,
$\widetilde{\lambda}_j({\mathcal A}^{{\mathcal V}_k}(\omega))$ at the iterates 
$\omega^{(k)}$ of Algorithm \ref{subspace_algorithm} for large $k$ in the way stated by 
Lemma \ref{thm:extended_Hermite_interpolate} below.
This result immediately follows from part (ii) of Lemma \ref{monotonicity_Hermite}, as
the set $\{ \widetilde{\lambda}_j(\mathcal{A}(\omega)) \; | \; j  = 1,\dots, K \}$ 
corresponds to the set of largest $K$ eigenvalues of ${\mathcal A}(\omega)$ for all $\omega$ 
in an open interval ${\mathcal I}$ containing $\omega_\ast$. Furthermore, $\omega^{(k)} \in {\mathcal I}$ 
for large $k$, and the eigenvectors corresponding to $\widetilde{\lambda}_j(\mathcal{A}(\omega^{(k)}))$ 
for $j = 1,\dots,K$ are included in the subspaces.
\begin{lemma}\label{thm:extended_Hermite_interpolate}
Let $\{ \omega^{(k)} \}$ be the sequence generated by Algorithm \ref{subspace_algorithm},
and $\omega_\ast$ denote the unique global minimizer of $\lambda_{\max}({\mathcal A}(\omega))$
over $\Omega$ such that $\lambda_{\max}({\mathcal A}(\omega_\ast))$ has algebraic
multiplicity $K \geq 2$. For each $k \in {\mathbb Z}^+$ large enough, there exist 
$\ell_1, \dots, \ell_K \in \{1, \dots,$ $\dim \: {\mathcal V}_k \}$ satisfying 
\[
	 \widetilde{\lambda}_j({\mathcal A} (\omega^{(k)}))	
					=
				\widetilde{\lambda}_{\ell_j}({\mathcal A}^{{\mathcal V}_k} (\omega^{(k)})),
			\quad \widetilde{\lambda}'_j({\mathcal A} (\omega^{(k)}))	
					=
				\widetilde{\lambda}'_{\ell_j}({\mathcal A}^{{\mathcal V}_k} (\omega^{(k)})),
\]
for $j = 1,\dots, K$.
\end{lemma}

It is a matter of convention how we label the real analytic eigenvalue functions of ${\mathcal A}^{{\mathcal V}_k}(\omega)$. 
For ease of notation, from here on, we relabel if necessary so that $\ell_j = j$ for $j = 1,\dots,K$ in Lemma \ref{thm:extended_Hermite_interpolate}.
The next lemma gives a description of $\omega^{(k+1)}$, the global minimizer of 
$\lambda_{\max}({\mathcal A}^{{\mathcal V}_k}(\omega))$, in terms of the real analytic
eigenvalues $\widetilde{\lambda}_1({\mathcal A}^{{\mathcal V}_k}(\omega)), \dots, 
				\widetilde{\lambda}_{K}({\mathcal A}^{{\mathcal V}_k}(\omega))$.

\begin{lemma}\label{intersection}
	Suppose that $\lambda_{\max}(\mathcal{A}(\omega))$ has a unique global minimizer
	over $\Omega$ at $\omega_\ast$. Furthermore, suppose $\lambda_{\max}({\mathcal A}(\omega_\ast))$ 
	has algebraic multiplicity $K \geq 2$, and that (\ref{eq:assumption_gen_der}) holds. The following assertions are 
	satisfied for all $k$ large enough:
	\begin{enumerate}
		\item [\bf (i)] Letting $\delta_j := \widetilde{\lambda}_j'(\mathcal{A}(\omega_\ast))>0$ for $j = 1,\dots,P$,
		there exists $\varepsilon_j  >0$ such that for all 
		$\omega \in (\omega_\ast-\varepsilon_j , \omega_\ast+\varepsilon_j)$
		we have 
		\[
			\widetilde{\lambda}_j'(\mathcal{A}(\omega)) \geq  3\delta_j/4	\quad	{\rm and}	\quad
			\widetilde{\lambda}_j'(\mathcal{A}^{\mathcal{V}_k}(\omega)) \geq \delta_j/2.
		\]
		\item [\bf (ii)] Letting $\delta_j := \widetilde{\lambda}_j'(\mathcal{A}(\omega_\ast)) < 0$ for $j = P+1,\dots,K$,
		there exists $\varepsilon_j  >0$ such that for all 
		$\omega \in (\omega_\ast-\varepsilon_j , \omega_\ast+\varepsilon_j)$
		we have 
		\[
			\widetilde{\lambda}_j'(\mathcal{A}(\omega)) \leq  3\delta_j/4	\quad	{\rm and}	\quad
			\widetilde{\lambda}_j'(\mathcal{A}^{\mathcal{V}_k}(\omega)) \leq \delta_j/2.
		\]
		\item[\bf (iii)] The point $\omega^{(k+1)}$ is such that 
			\[
				\lambda_{\max}(\mathcal{A}^{\mathcal{V}_k}(\omega^{(k+1)})) \;\; = \;\;
				\widetilde{\lambda}_{j_1}(\mathcal{A}^{\mathcal{V}_k}(\omega^{(k+1)})) \;\; = \;\;
						\widetilde{\lambda}_{j_2}(\mathcal{A}^{\mathcal{V}_k}(\omega^{(k+1)}))
			\]
			for some $j_1 \in \{ 1, \dots, P \}$ and some $j_2 \in \{ P+1, \dots, K \}$.
	\end{enumerate}
\end{lemma}
\begin{proof}
	\textbf{(i)} It follows from \cite[Proposition 2.9]{kangal2018} that there exists an open interval 
	$\mathcal{I}$ containing $\omega_\ast$ and a positive constant $\nu \in {\mathbb R}^+$ such that 
	$\widetilde{\lambda}_j''(\mathcal{A}(\omega))$, as well as 
	$\widetilde{\lambda}_j''(\mathcal{A}^{\mathcal{V}_k}(\omega))$ for all $k$ sufficiently large are 
	bounded in absolute value by $\nu$ uniformly over all $\omega \in \mathcal{I}$ and over all such
	large $k$.
	
	The uniform boundedness of $| \widetilde{\lambda}_j''(\mathcal{A}(\omega)) |$
	combined with $\widetilde{\lambda}_j'(\mathcal{A}(\omega_\ast)) = \delta_j$ imply the
	existence of an open interval 
	$\widehat{\mathcal{I}}:=(\omega_\ast-\hat{\varepsilon},\omega_\ast+\hat{\varepsilon})\subseteq\mathcal{I}$ 
	such that $\widetilde{\lambda}_1'(\mathcal{A}(\omega)) \geq 3\delta_j/4$ $\: \forall \omega\in\widehat{\mathcal{I}}$. 
	
	Since the global minimizer $\omega_\ast$ is assumed to be unique, by Lemma \ref{convergence_omega} 
	we have $\omega^{(k)}\rightarrow\omega_\ast$ as $k \rightarrow \infty$. 	
	Hence, choose $k$ large enough so that 
	$\omega^{(k)} \in \widehat{\mathcal{I}}/2 := 
		(\omega_\ast-\hat{\varepsilon}/2,\omega_\ast+\hat{\varepsilon}/2)$.  
	Now we employ the Hermite interpolation property (Lemma \ref{thm:extended_Hermite_interpolate}) 
	to deduce
	\[
		\widetilde{\lambda}_j'(\mathcal{A}^{\mathcal{V}_k}(\omega^{(k)}))
				\; = \;
		\widetilde{\lambda}_j'(\mathcal{A}(\omega^{(k)}))
				\; \geq \;
				3\delta_j/4
	\]
	for all such large $k$. Since $| \widetilde{\lambda}_j''(\mathcal{A}^{\mathcal{V}_k}(\omega)) |$ 
	is also uniformly bounded in $\widehat{\mathcal{I}}$ by a constant independent of $k$, there exists 
	$\widetilde{\mathcal{I}}:=
		(\omega^{(k)}-\widetilde{\varepsilon},\omega^{(k)}+\widetilde{\varepsilon}) \subseteq \widehat{\mathcal I}$,
	in particular an $\widetilde{\varepsilon} \in {\mathbb R}^+$ independent of $k$, such that 
	\[
		\widetilde{\lambda}_j'(\mathcal{A}^{\mathcal{V}_k}(\omega))
					\; \geq \; \delta_j/2 \quad \forall \omega \in \widetilde{\mathcal{I}}
	\]
	for all such large $k$. By choosing $k$ even larger if necessary, it can be ensured that 
	$\abs{\omega^{(k)}-\omega_\ast} < \widetilde{\varepsilon}/2$ so that 
	$(\omega_\ast - \varepsilon_j,\omega_\ast+\varepsilon_j) \subseteq \widetilde{\mathcal{I}}$ for 
	$\varepsilon_j := \widetilde{\varepsilon}/2$. 
	
	\textbf{(ii)} This can be proven in a way similar to part (i).
	
	\textbf{(iii)} Setting $\varepsilon := \min \{ \varepsilon_j \; | \; j = 1,\dots, K \}$ for $\varepsilon_j$
	as in part (i) and (ii), there exists an open interval 
	$\mathcal{I}=(\omega_\ast-\varepsilon,\omega_\ast+\varepsilon)$ such that for all large $k$
	the following hold:
	\begin{equation*}
		\begin{split}
		\widetilde{\lambda}_j'(\mathcal{A}^{\mathcal{V}_k}(\omega)) \; & = \; \delta_j/2 \; > \; 0 
		\;\; \forall \omega \in {\mathcal I}
		\quad\quad {\rm for} \; j = 1,\dots, P,	\\
		\widetilde{\lambda}_j'(\mathcal{A}^{\mathcal{V}_k}(\omega)) \; & = \; \delta_j/2 \; < \; 0
		\;\; \forall \omega \in {\mathcal I}
		\quad\quad {\rm for} \; j = P+1, \dots, K.
		\end{split}
	\end{equation*}
	Furthermore, let $\widetilde{\mathcal I} := (\omega_\ast-\widetilde{\varepsilon},\omega_\ast+\widetilde{\varepsilon})$
	for $\widetilde{\varepsilon} := \min \{ \varepsilon , \varphi/ (8 \eta) \}$, where $\eta$ is the uniform Lipschitz
	constant in Lemma \ref{thm:Lip_cont}, and $\varphi$ is the eigenvalue gap as in (\ref{eq:eig_gap}).
	Without loss of generality, we make the following
	two assumptions. First, the set $\{ \widetilde{\lambda}_1({\mathcal A}(\omega)), \dots, \widetilde{\lambda}_K({\mathcal A}(\omega)) \}$
	corresponds to the set of $K$ largest eigenvalues of ${\mathcal A}(\omega)$ for all $\omega \in \widetilde{\mathcal I}$. 
	Secondly, $ \omega^{(k)} $,
	$
		\omega^{(k+1)} :=
		\arg\min_{\omega\in\Omega} \lambda_{\max}(\mathcal{A}^{\mathcal{V}_k}(\omega))$ $\in \widetilde{\mathcal I},
	$
	since $\omega^{(k)} \rightarrow \omega_\ast$ as $k \rightarrow \infty$.
	
	We start by showing that 
	$\lambda_{\max}({\mathcal A}^{{\mathcal V}_k}(\omega)) > \widetilde{\lambda}_{j'}({\mathcal A}^{{\mathcal V}_k}(\omega))$
	for all $\omega \in \widetilde{\mathcal I}$ and for $j' = K+1, \dots, \dim \: {\mathcal V}_k$. To this end, 
	first observe that
	\[
		\lambda_K({\mathcal A}(\omega))	-	\lambda_{K+1}({\mathcal A}(\omega))
						\;\; \geq \;\;
		\lambda_K({\mathcal A}(\omega_\ast))	-	\lambda_{K+1}({\mathcal A}(\omega_\ast))	- 	\varphi/4
						\;\; = \;\;	3 \varphi/4
	\]
	for all $\omega \in \widetilde{\mathcal I}$, where we employ the Lipschitz continuity 
	of $\lambda_K({\mathcal A}(\omega))$ and $\lambda_{K+1}({\mathcal A}(\omega))$ with the Lipschitz
	constant $\eta$ (see Lemma \ref{thm:Lip_cont}). 
	In particular, for each $j \in \{1, \dots, K \}$, we have
	\[
		\begin{split}
		3 \varphi/4		\;\; & \leq \;\;
		\widetilde{\lambda}_{j}({\mathcal A}(\omega^{(k)}))	-	\lambda_{K+1}({\mathcal A}(\omega^{(k)}))	\\
					\;\;& \leq \;\;
		\widetilde{\lambda}_{j}({\mathcal A}^{{\mathcal V}_k}(\omega^{(k)}))	-   \lambda_{K+1}({\mathcal A}^{{\mathcal V}_k}(\omega^{(k)}))
		\end{split}
	\]
	where in the second line we exploit the interpolation property (Lemma \ref{thm:extended_Hermite_interpolate}),
	and the monotonicity (i.e., $\lambda_{K+1}({\mathcal A}(\omega^{(k)})) \geq \lambda_{K+1}({\mathcal A}^{{\mathcal V}_k}(\omega^{(k)}))$
	due to part (i) of Lemma \ref{monotonicity_Hermite}).
	Now the last inequality implies
	\begin{equation}\label{eq:inter_eigval_gap}
		\begin{split}
		\widetilde{\lambda}_{j}({\mathcal A}^{{\mathcal V}_k}(\omega))	-   \lambda_{K+1}({\mathcal A}^{{\mathcal V}_k}(\omega))
					\;\; \geq \hskip 25ex	\\
		\widetilde{\lambda}_{j}({\mathcal A}^{{\mathcal V}_k}(\omega^{(k)}))	-   \lambda_{K+1}({\mathcal A}^{{\mathcal V}_k}(\omega^{(k)}))
					-
				\varphi/2		\;\; = \;\;	\varphi/4	
		\end{split}
	\end{equation}
	for all $\omega \in \widetilde{\mathcal I}$. Note that above we make use of the uniform Lipschitz continuity of
	$\widetilde{\lambda}_{j}({\mathcal A}^{{\mathcal V}_k}(\omega))$ and $\lambda_{K+1}({\mathcal A}^{{\mathcal V}_k}(\omega))$
	with the uniform Lipschitz constant $\eta$ independent of the subspace ${\mathcal V}_k$; 
	the latter is immediate from Lemma \ref{thm:Lip_cont}, whereas the former can be seen from 
	$|\widetilde{\lambda}'_{j}({\mathcal A}^{{\mathcal V}_k}(\omega))| \leq \| {\mathcal A}'(\omega) \|_2$.
	(Strictly speaking $\eta$ as in Lemma \ref{thm:Lip_cont} is the uniform Lipschitz constant for the
	sorted eigenvalue $\lambda_{j}({\mathcal A}^{{\mathcal V}_k}(\omega))$, but without loss of generality
	it can be chosen even larger if necessary so that it is also at least as large as
	the uniform Lipschitz constant for $\widetilde{\lambda}'_{j}({\mathcal A}^{{\mathcal V}_k}(\omega))$.) 
	Inequality (\ref{eq:inter_eigval_gap}) means that $\lambda_{K+1}({\mathcal A}^{{\mathcal V}_k}(\omega))$
	is the largest of $\widetilde{\lambda}_{j}({\mathcal A}^{{\mathcal V}_k}(\omega))$ for $j = K+1,\dots, \dim \: {\mathcal V}_k$
	for $\omega \in \widetilde{\mathcal I}$, so for every $j = K+1, \dots, N$ for all $\omega \in \widetilde{\mathcal I}$,
	we deduce
	\[
		\lambda_{\max} ({\mathcal A}^{{\mathcal V}_k}(\omega))		-	\widetilde{\lambda}_{j'}({\mathcal A}^{{\mathcal V}_k}(\omega))
						\;\; \geq \;\;
		\widetilde{\lambda}_{j}({\mathcal A}^{{\mathcal V}_k}(\omega))	-   \lambda_{K+1}({\mathcal A}^{{\mathcal V}_k}(\omega))	
						\;\; \geq \;\;		\varphi/4.
	\]
	
	It follows from the previous paragraph that 
			$\lambda_{\max}({\mathcal A}^{{\mathcal V}_k}(\omega)) = \widetilde{\lambda}_{j}({\mathcal A}^{{\mathcal V}_k}(\omega))$
	for some $j \in \{ 1, \dots, K \}$ for all $\omega \in \widetilde{\mathcal I}$, yet
			$\lambda_{\max}({\mathcal A}^{{\mathcal V}_k}(\omega)) > \widetilde{\lambda}_{j}({\mathcal A}^{{\mathcal V}_k}(\omega))$
	for $j \notin \{ 1, \dots, K \}$.
	Now consider a point $\widetilde{\omega} \in \widetilde{\mathcal I}$ such that 
	\begin{itemize}
		\item $\lambda_{\max}(\mathcal{A}^{\mathcal{V}_k}(\widetilde{\omega})) = 
					\widetilde{\lambda}_{p_1}(\mathcal{A}^{\mathcal{V}_k}(\widetilde{\omega}))	=
						\dots		=
					\widetilde{\lambda}_{p_q}(\mathcal{A}^{\mathcal{V}_k}(\widetilde{\omega}))$  \\
			$\exists p_1, \dots, p_q \in$ $\{1, \dots, P \}$ for some $q \geq 1$, \\
			yet
			$\lambda_{\max}(\mathcal{A}^{\mathcal{V}_k}(\widetilde{\omega})) >
				\widetilde{\lambda}_j(\mathcal{A}^{\mathcal{V}_k}(\widetilde{\omega}))$ for $j \not\in \{ p_1, \dots, p_q \}$, or
		\item	$\lambda_{\max}(\mathcal{A}^{\mathcal{V}_k}(\widetilde{\omega})) = 
					\widetilde{\lambda}_{n_1}(\mathcal{A}^{\mathcal{V}_k}(\widetilde{\omega}))	=
						\dots		=
					\widetilde{\lambda}_{n_s}(\mathcal{A}^{\mathcal{V}_k}(\widetilde{\omega}))$ \\
			$\quad
			\exists n_1, \dots, n_s \in \{P + 1, \dots, K \}$ for some $s \geq 1$, \\
			yet
			$\lambda_{\max}(\mathcal{A}^{\mathcal{V}_k}(\widetilde{\omega})) >
				\widetilde{\lambda}_j(\mathcal{A}^{\mathcal{V}_k}(\widetilde{\omega}))$ for $j \notin \{ n_1, \dots, n_s \}$.
	\end{itemize}
	We shall show that such a point cannot be a minimizer of $\lambda_{\max}(\mathcal{A}^{\mathcal{V}_k}(\omega))$.
	For these two cases, we respectively have
	\begin{equation*}
		\begin{split}
		\partial \lambda_{\max}(\mathcal{A}^{\mathcal{V}_k}(\widetilde{\omega}))
				& =
		{\rm Co} \left\{ \widetilde{\lambda}'_{p_j}(\mathcal{A}^{\mathcal{V}_k}(\widetilde{\omega})) \; | \; j = 1,\dots, q \right\} 
				\;\; \subseteq \;\; (\delta_{p,\min}/2 , \infty), \\
		\partial \lambda_{\max}(\mathcal{A}^{\mathcal{V}_k}(\widetilde{\omega}))
				& =
		{\rm Co} \left\{ \widetilde{\lambda}'_{n_j}(\mathcal{A}^{\mathcal{V}_k}(\widetilde{\omega})) \; | \; j = 1,\dots, s \right\} 
				\;\; \subseteq \;\; (-\infty, \delta_{n,\max}/2) 
		\end{split}
	\end{equation*}
	where $\delta_{p,\min} := \min\{ \delta_{p_j} \; | \; j = 1,\dots, q \} > 0$, and 
	$\delta_{n,\max} := \max\{ \delta_{n_j} \; | \; j = 1,\dots, s \} < 0$. In either case,
	$0 \notin \partial \lambda_{\max}(\mathcal{A}^{\mathcal{V}_k}(\widetilde{\omega}))$ implying
	$\widetilde{\omega}$ cannot be a minimizer of $\lambda_{\max}(\mathcal{A}^{\mathcal{V}_k}(\omega))$.
	Therefore, since $\omega^{(k+1)} \in \widetilde{\mathcal I}$ is a minimizer of 
	$\lambda_{\max}(\mathcal{A}^{\mathcal{V}_k}(\omega))$, we must have 
	$\lambda_{\max}(\mathcal{A}^{\mathcal{V}_k}(\omega^{(k+1)}))
			=
		\widetilde{\lambda}_{j_1}(\mathcal{A}^{\mathcal{V}_k}(\omega^{(k+1)}))
			=	\widetilde{\lambda}_{j_2}(\mathcal{A}^{\mathcal{V}_k}(\omega^{(k+1)}))$
	for some $j_1 \in \{ 1, \dots, P \}$ and $j_2 \in \{ P+1, \dots, K \}$.
\end{proof}

Now we are ready to present the main quadratic rate-of-convergence result in the non-smooth setting;
this result follows from the Hermite interpolation properties in Lemma \ref{thm:extended_Hermite_interpolate},
as well as Lemma \ref{intersection}.
\begin{theorem}[Quadratic Convergence in the Non-smooth Case]
	Suppose that the global minimizer 
		$\omega_\ast := \arg \min_{\omega \in \Omega} \: \lambda_{\max}({\mathcal A}(\omega))$
	is unique and such that the eigenvalue $\lambda_{\max}({\mathcal A}(\omega_\ast))$ is multiple,
	condition (\ref{eq:assumption_gen_der}) holds. The sequence 
	$\left\{\omega^{(k)}\right\}$ generated by Algorithm  \ref{subspace_algorithm} satisfies
	\begin{equation}\label{eq:quad_convergence}
			| \omega^{(k+1)} - \omega_\ast |	\; = \;	O((\omega^{(k)} - \omega_\ast)^2)
	\end{equation}
	for all large $k$.
\end{theorem}
\begin{proof}
	Part (i) an (ii) of Lemma \ref{intersection} shows the existence of an open interval 
	$\mathcal{I} := (\omega_\ast - \varepsilon,\omega_\ast+\varepsilon)$ where 
	$\widetilde{\lambda}_j'(\mathcal{A}(\omega))$ for $j = 1,\dots, P$ is bounded from below uniformly 
	by a positive real number and $\widetilde{\lambda}_j'(\mathcal{A}(\omega))$ for $j = P+1, \dots, K$
	is bounded from above uniformly by a negative real number. Let us consider $k$ large enough
	so that $\omega^{(k)} \in {\mathcal I}$. 
	
	By part (iii) of Lemma \ref{intersection}, there exist
	$j_1 \in \{1, \dots, P \}$ and $j_2 \in \{ P+1, \dots, K \}$ such that 
	$\widetilde{\lambda}_{j_1}(\mathcal{A}^{\mathcal{V}_k}(\omega^{(k+1)})) \;\; = \;\;
						\widetilde{\lambda}_{j_2}(\mathcal{A}^{\mathcal{V}_k}(\omega^{(k+1)}))$.
	For such a pair of $j_1$, $j_2$ define the real analytic functions
	\begin{equation*}
			\begin{split}
			\lambda({\mathcal A}(\omega)) \;\; & := \;\; 
				\widetilde{\lambda}_{j_1}({\mathcal A}(\omega))		-	
						\widetilde{\lambda}_{j_2}({\mathcal A}(\omega))	,	\\
			\lambda({\mathcal A}^{{\mathcal V}_k}(\omega)) \;\; & := \;\; 
				\widetilde{\lambda}_{j_1}({\mathcal A}^{{\mathcal V}_k}(\omega)) 	-  
						\widetilde{\lambda}_{j_2}({\mathcal A}^{{\mathcal V}_k}(\omega)).
			\end{split}		
	\end{equation*}
	Observe that there exists a constant $\zeta \in {\mathbb R}^+$ such that
	\[
		\lambda'({\mathcal A}(\omega^{(k)}))
			 = 
		\widetilde{\lambda}_{j_1}'(\mathcal{A}(\omega^{(k)})) 
					- 
		\widetilde{\lambda}_{j_2}'(\mathcal{A}(\omega^{(k)})) \; \geq \; \zeta
		\;	\Longrightarrow		\;
		\left|  \left[ \lambda'({\mathcal A}(\omega^{(k)})) \right]^{-1} \right|	\; \leq \; \zeta^{-1}
	\]
	independent of $k$. Note also that the real analyticity of $\widetilde{\lambda}_{j_1}(\mathcal{A}(\omega))$,
	$\widetilde{\lambda}_{j_2}(\mathcal{A}(\omega))$ implies the Lipschitz continuity of $\lambda'({\mathcal A}(\omega))$
	on ${\mathcal I}$.
	
	The proof manipulates the following equation, which is immediate from an application of Taylor's theorem 
	with integral remainder:
	\[
		0 \; = \;  \lambda({\mathcal A}(\omega_\ast)) \; = \;
			\lambda({\mathcal A}(\omega^{(k)})) + 
		\int_{0}^{1}\lambda'({\mathcal A}(\omega^{(k)} + t(\omega_\ast-\omega^{(k)}))) (\omega_\ast - \omega^{(k)}) \: dt.
	\]
	We employ $\lambda({\mathcal A}(\omega^{(k)}))=\lambda({\mathcal A}^{{\mathcal V}_k}(\omega^{(k)}))$ 
	(due to Lemma \ref{thm:extended_Hermite_interpolate}) in this equation, then multiply both sides 
	by $\left[ \lambda'({\mathcal A}(\omega^{(k)})) \right]^{-1}$ to obtain
	\begin{equation}\label{taylor1}
		\begin{split}
		0 	\; = \; 
		\left[ \lambda'({\mathcal A}(\omega^{(k)}))\right]^{-1}
			\lambda ({\mathcal A}^{{\mathcal V}_k}(\omega^{(k)}))   +  (\omega_\ast - \omega^{(k)})	 + 
		\left[ \lambda'({\mathcal A}(\omega^{(k)}))\right]^{-1} \times \\
		\int_{0}^{1} \left[\lambda'({\mathcal A}(\omega^{(k)} + t(\omega_\ast - \omega^{(k)}))) - 
													\lambda'({\mathcal A}(\omega^{(k)})) \right]
						(\omega_\ast - \omega^{(k)}) \: dt.
		\end{split}
	\end{equation}
	An application of Taylor's theorem to $\lambda({\mathcal A}^{{\mathcal V}_k}(\omega))$ about $\omega^{(k)}$
	with second order remainder combined with the equalities
	$\lambda({\mathcal A}^{{\mathcal V}_k}(\omega^{(k+1)})) = 0$ and 
	$\lambda'({\mathcal A}^{{\mathcal V}_k}(\omega^{(k)})) = \lambda'({\mathcal A}(\omega^{(k)}))$
	(a corollary of Lemma \ref{thm:extended_Hermite_interpolate}) lead us to
	\[
			\left[ \lambda'({\mathcal A}(\omega^{(k)})) \right]^{-1}\lambda({\mathcal A}^{{\mathcal V}_k}(\omega^{(k)})) 
					\; = \;  (\omega^{(k)}-\omega^{(k+1)}) + O((\omega^{(k)}-\omega^{(k+1)})^2).
	\]
	Now using the last equality in (\ref{taylor1}) gives rise to
	\begin{equation*}
	\begin{split}
		0	\; = \; 	(\omega_\ast - \omega^{(k+1)}) \; + \; O((\omega^{(k)}-\omega^{(k+1)})^2) \; + \;
		\left[ \lambda'({\mathcal A}(\omega^{(k)})) \right]^{-1} \times \hskip 3ex \\
			\hskip 2ex \int_{0}^{1}\left[\lambda'({\mathcal A}(\omega^{(k)} + t(\omega_\ast - \omega^{(k)}))) - 
								\lambda'({\mathcal A}(\omega^{(k)}))\right](\omega_\ast - \omega^{(k)}) \: dt,
	\end{split}
	\end{equation*}
	implying
	\begin{equation}\label{eq:quad_convergence_imed}
		\begin{split}
		\abs{\omega^{(k+1)}-\omega_\ast} \; \leq \; O((\omega^{(k)}-\omega^{(k+1)})^2) \; + \;
		\abs{\left[ \lambda'({\mathcal A}(\omega^{(k)})) \right]^{-1}}	\times \hskip 6ex \\
		\hskip 4ex  \int_{0}^{1}\abs{\lambda'({\mathcal A}(\omega^{(k)} + t(\omega_\ast - \omega^{(k)}))) 
							- \lambda'({\mathcal A}(\omega^{(k)}))} \abs{\omega_\ast - \omega^{(k)}} \: dt.
		\end{split}
	\end{equation}
	Now the desired equality (\ref{eq:quad_convergence}) follows from (\ref{eq:quad_convergence_imed})
	by employing the bound 
	$ \left|  \left[  \lambda'({\mathcal A}(\omega^{(k)}))  \right]^{-1} \right| \leq \zeta^{-1}$,
	the Lipschitz continuity of $\lambda'({\mathcal A}(\omega))$,
	as well as the inequality
	\[
		(\omega^{(k)}-\omega^{(k+1)})^2
				\; \leq \;
		2 \left[    (\omega^{(k+1)}-\omega_\ast)^2+(\omega^{(k)}-\omega_\ast)^2   \right].
	\]	
\end{proof}

\subsection{Numerical Experiments}
In this section we test Algorithm \ref{subspace_algorithm} on numerical examples that concern
three applications, namely the computation of the distance to a nearest definite pair, determining whether a given 
quadratic eigenvalue problem is hyperbolic or not, and inferring the existence of a conjugate gradient iteration 
for the saddle point linear system. In particular, the quadratic convergence of the subspace procedure is illustrated 
in the non-smooth case on examples arising from the latter two applications. 

The subspace framework is terminated when the condition
\[
		\abs{\lambda_{\max}(\mathcal{A}^{\mathcal{V}_k}(\omega^{k+1}))
			-
		\lambda_{\max}(\mathcal{A}^{\mathcal{V}_{k-1}}(\omega^{k}))}<\texttt{tol}
\] 
is satisfied, where \texttt{tol} is a prescribed tolerance; we always set $\texttt{tol}=10^{-12}$ unless otherwise 
specified. Additionally, in all examples, the input parameter $\epsilon$ determining the maximal separation of 
an eigenvalue from the largest eigenvalue for the inclusion of the corresponding eigenvector in the subspace
is set equal to $10^{-6}$. The reduced eigenvalue optimization problems are solved by means of the 
algorithms discussed in Section \ref{small_scale_algorithms}. Indeed, we typically employ the support 
based algorithm (see Section \ref{sec:support_algorithm}), in particular the MATLAB package 
\texttt{eigopt} \cite[Section 10]{Mengi2014}, which is an implementation of this algorithm. 
The package requires a box to be optimized over, and a global lower bound $\gamma$
on the second derivatives of the eigenvalue function; we provide the natural and theoretically
sound choices for these, namely $[0,2\pi]$ for the box and $-\norm{A}_2-\norm{B}_2$ for the
lower bound on the second derivatives. \\

\noindent
\textbf{Distance to the nearest definite pair.} 
Consider the Hermitian matrix pair $(A,B)$ with
\[
	A \; = \;\frac{\widetilde{G}+\widetilde{G}^\ast}{2} 
			\qquad 	\text{and} 		\qquad 
	B \; = \; -\ri \frac{\widetilde{G}-\widetilde{G}^\ast}{2}
\]
and $\widetilde{G}:=Ge^{\ri \pi/6}$, where $G$ is the $640\times640$ Grcar matrix. An application of the 
subspace framework for the computation of the inner numerical radius of $A+ \ri B$ yields the global minimizer 
$\theta_\ast = 2.617992877994$ of $\lambda_{\max}(A\cos\theta+B\sin\theta)$, as well as 
$\lambda_{\max}(A\cos\theta_\ast+B\sin\theta_\ast)=0.634045490256$. Now Theorem \ref{char_inner}, 
implies $0\in F(A + \ri B)$; this is confirmed by the plot of $F(A + \ri B)$ given on the top left in 
Figure \ref{fig:large_scale_dist}. Consequently, the pair $(A,B)$ is indefinite.

We deduce from Theorem \ref{solving_distance} that
$
	d_\delta(A,B)=0.644045490256
$
for $\delta = 10^{-2}$. Figure \ref{fig:large_scale_dist} illustrates the field of values of $A + \ri B$, 
$(A+\Delta A) + \ri (B+\Delta B)$ and $\widetilde{A} + \ri \widetilde{B} =e^{-\ri \varphi}((A+\Delta A) + \ri (B+\Delta B))$
for the choice of $\delta = 10^{-2}$, where $\Delta A,\Delta B$, $\varphi$ are as in Algorithm \ref{alg_dist}. 
The diamonds in the figure correspond to the points where the inner numerical radii are attained. 
Although $\lambda_{\max}(A\cos\theta_\ast + B\sin\theta_\ast)$ is simple, it turns out to be very close to the second 
largest eigenvalue; these two eigenvalues differ by an amount on  the order of $10^{-7}$. 
Quadratic convergence is achieved by our subspace framework; this is depicted in Table \ref{table_dist},
which lists the iterates of Algorithm \ref{subspace_algorithm} with $\epsilon=10^{-6}$.
\begin{figure}[h]
	$\begin{array}{rl}
	\includegraphics[width=6.2cm]{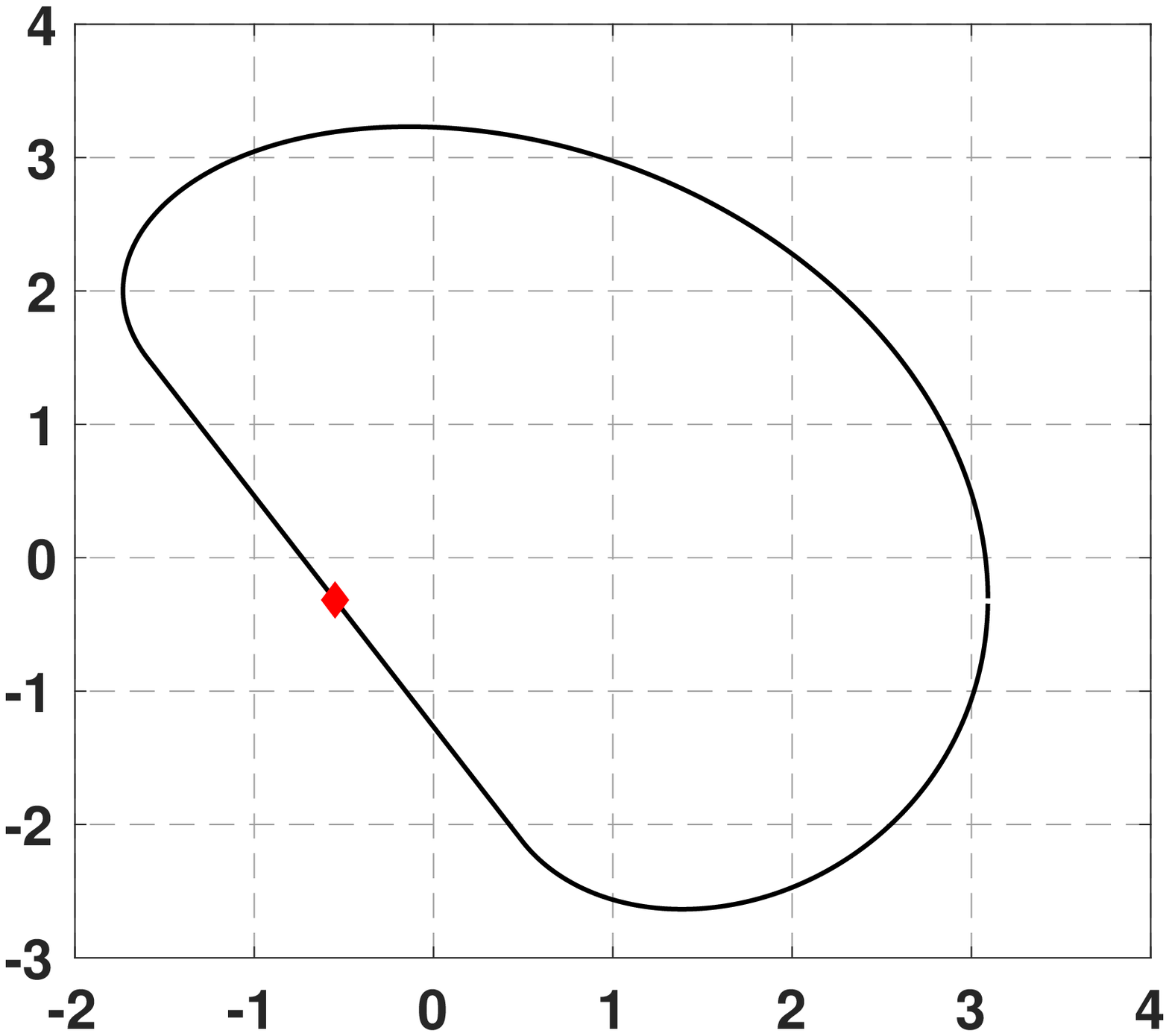} &
	\includegraphics[width=6.2cm]{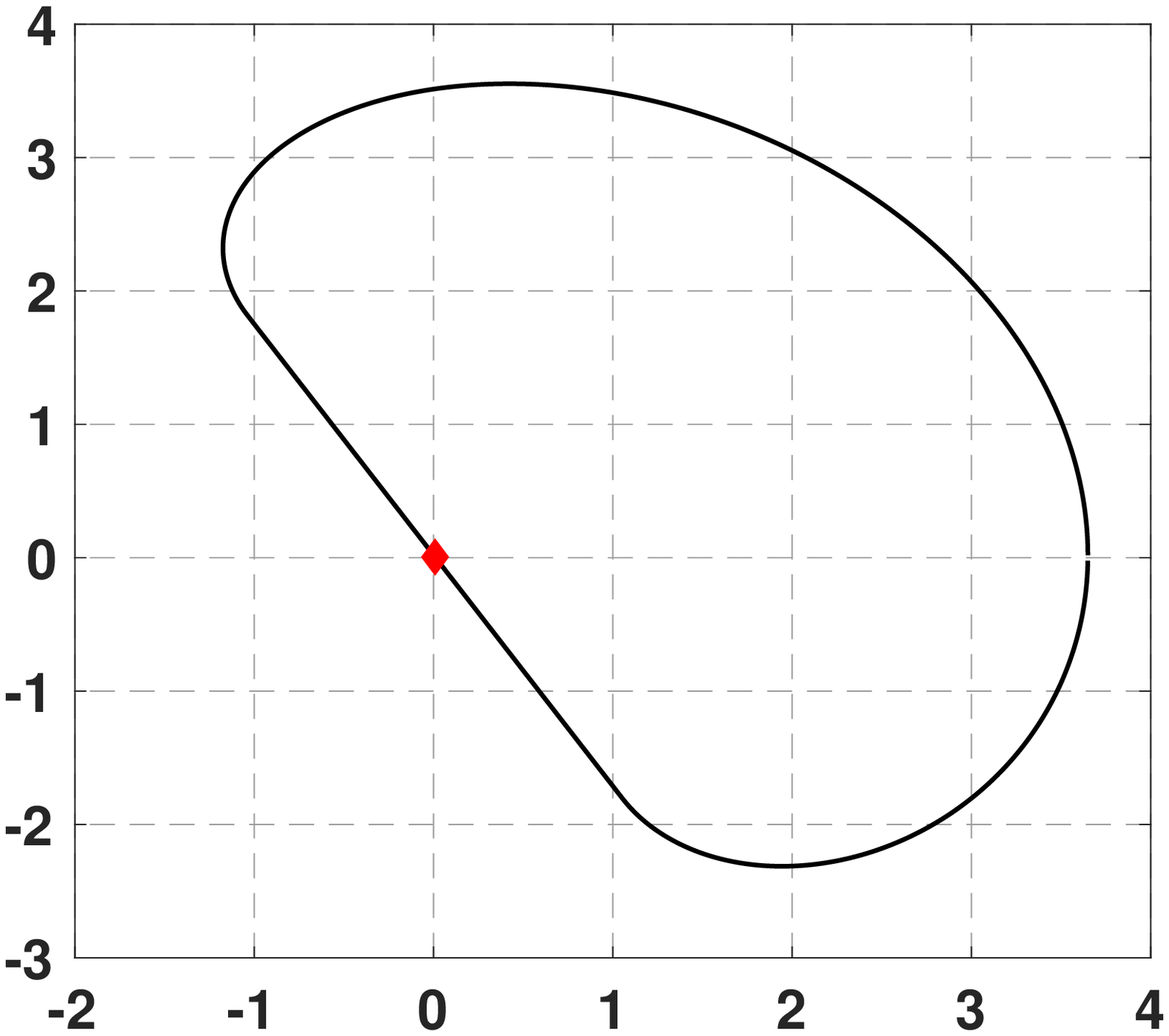}\\
	\multicolumn{2}{c}{\includegraphics[width=6.2cm]{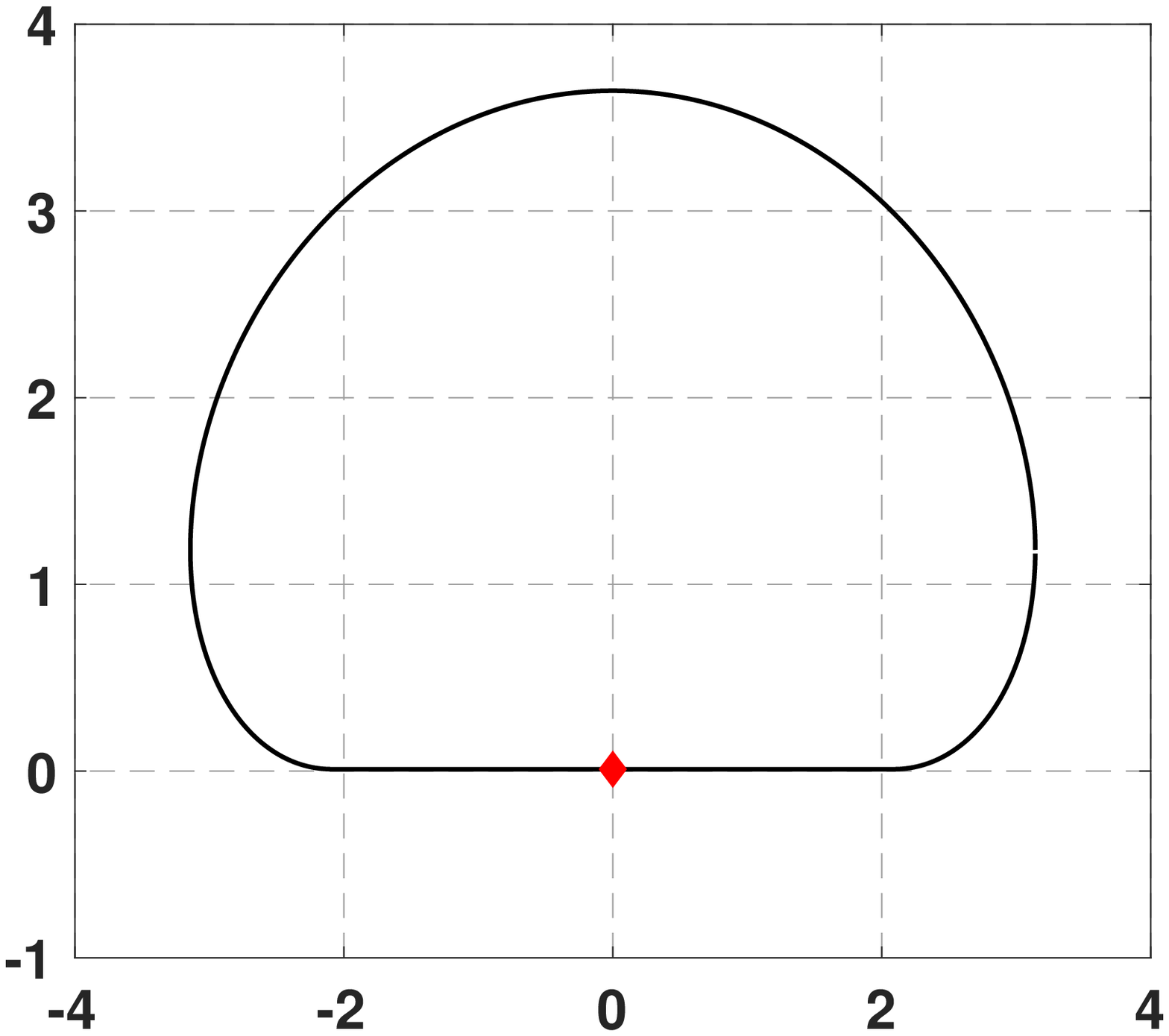}}
	\end{array}$
	\caption{The figure depicts the field of values of $A+iB$ (top left), $A+\Delta A+ i(B+\Delta B)$ (top right), 
	$e^{-i\psi}(A+\Delta A +{\rm i}(B+\Delta B))$ (bottom) for the choice of $A = (G + G^\ast)/2$ and
	$B = -\ri (G - G^\ast)/2$ with $G$ denoting the $640\times 640$ Grcar matrix.
	The minimal perturbation matrices $\Delta A, \Delta B$ and the quantity $\psi$ are 
	those returned by Algorithm \ref{alg_dist} for $\delta = 10^{-2}$. The diamonds mark 
	the points where the inner numerical radii are attained.} 
	\centering
	\label{fig:large_scale_dist}
\end{figure}
\begin{table}[tb]
	\begin{center}
		\begin{tabular}{ c| c | c | c }
			\hline
			\hline
			$k$ & $p$ & $\lambda_{\max}({\mathcal A}^{{\mathcal V}_k}({\mathcal \omega}^{(k+1)}))$ & $\omega^{(k+1)}$ \\
			\hline
			4 & 4 & 0.629840138568 &2.618127739876 \\
			5 & 5 & 0.\underline{63}2130046510 & 2.\underline{617}077493245  \\
			6 & 6 & 0.\underline{634045}279755 & 2.\underline{617993}771783\\
			7 &  8& 0.\underline{634045490256} & 2.\underline{617993877986} \\
			8 & 10 & 0.\underline{634045490256}  & 2.\underline{617993877986}\\
			\hline
			\hline
		\end{tabular}
	\end{center}
	\caption{This table concerns an application of the subspace framework to the example of Figure \ref{fig:large_scale_dist}, 
	which involves the minimization of the largest eigenvalue of ${\mathcal A}(\omega) = A \cos \omega + B \sin \omega$ 
	over $\omega$ for certain Hermitian matrices $A, B$ defined in terms of the $640\times 640$ Grcar matrix. 
	The globally minimal value of the reduced eigenvalue function 
	$\lambda_{\max}(\mathcal{A}^{\mathcal{V}_k}(\omega^{(k+1)}))$ 
	and $\omega^{(k+1)} := \arg\min_{\omega} \lambda_{\max}({\mathcal A}^{{\mathcal V}_k}(\omega))$
	by Algorithm \ref{subspace_algorithm} starting with $\omega^{(1)} = 0.45$ are listed along with the subspace 
	dimension $p:= {\rm dim} \: \mathcal{V}_k$ with respect to the iteration number $k$.}
	\label{table_dist}
\end{table}

\medskip

\noindent
\textbf{Testing hyperbolicity of QEP.} 
We consider the quadratic eigenvalue problem $Q(\lambda)=\lambda^2 A +\lambda B +C$ that is linked to
a damped-mass spring system, where $A,B,C$ are $500\times 500$ matrices such that $A = I$,
\begin{equation}\label{hyper_example}
	B = 
		\beta
		\left[
			\begin{array}{ccccc}
				20 & -10 & & & \\
				-10 & 30 & \ddots & & \\
				     & \ddots & \ddots & \ddots & \\
				     & & \ddots & 30 & -10 \\
				     & & & -10 & 20
			\end{array}
		\right],
		\;
	C =
		\left[
			\begin{array}{cccc}
				15 & -5 & & \\
				-5 & 15 & \ddots & \\
				    & \ddots & \ddots & \\
				    & & -5 & 15 
			\end{array}
		\right]
\end{equation}
for a given real number $\beta>0$. We construct the coefficient matrices of this QEP by using the 
MATLAB toolbox NLEVP \cite{betcke2013nlevp}, and determine the hyperbolicity of the QEP for different $\beta$ 
values by testing the definiteness of the associated $1000\times1000$ pair $(A_1,B_1)$ as in (\ref{hyper_pair}). 
Table \ref{test_hyper} indicates whether the pair $(A_1,B_1)$ is definite or not for eight equally spaced $\beta$ values 
in the interval $[0.500, 0.528]$. The number of iterations to compute the inner numerical radius
of $A_1+ \ri B_1$ (up to the prescribed tolerance \texttt{tol}$=10^{-12}$) is eight for each $\beta$ 
value. This is a non-smooth example; the largest eigenvalue of $A_1\cos\theta+B_1\sin\theta$ has multiplicity $2$ 
at the minimizing $\theta$ for each $\beta$ value, yet we observe the quadratic convergence of 
Algorithm \ref{subspace_algorithm} consistent with what is expected in theory.
This is hinted by Table \ref{QEP_mult}, which lists the iterates $\lambda_{\max}({\mathcal A}^{{\mathcal V}_k}(\omega^{(k+1)}))$
and $\omega^{(k+1)}$ with respect to the iteration number $k$ for $\beta=0.512$ and $\beta=0.524$. 
\begin{table}[tb]
	\begin{center}
		\begin{tabular}{|c| c | c | c | c | c | c | c | c |}
			\hline
			$\beta$ & $0.500$ & $0.504$ & $0.508$ & $0.512$ & $0.516$ & $0.520$ & $0.524$ & $0.528$  \\
			\hline
			definite & no & no & no & no & no & yes & yes & yes \\
			\hline
		\end{tabular}
	\end{center}
	\caption{The definiteness of the pair $(A_1,B_1)$ defined as in (\ref{hyper_pair}) associated with the quadratic
	eigenvalue problem $Q(\lambda) = \lambda^2 A + \lambda B +C$ for $A, B, C$ as in (\ref{hyper_example}) and
	for several values of $\beta$.}
	\label{test_hyper}
\end{table}

\begin{table}[tb]	

	\begin{center}
	
		\begin{tabular}{cccc}
			\hline
			\hline
			\multicolumn{4}{c}{{\sc $\beta=0.512$}}			\hskip 20ex								 \\
			$k$	&	$p$ & $\lambda_{\max}({\mathcal A}^{{\mathcal V}_k}(\omega^{(k+1)}))$  & $\omega^{(k+1)}$  		\\
			\hline
			4		&	6	&	0.218717671040			& 1.972911641774  						\\
			5		&	8	&	0.\underline{00}6821928930	&  1.\underline{89}0604000858 			 \\
			6		&	10	&	0.\underline{008594}146027	& 1.\underline{8971}61234772 			 	\\
			7		&	12	&	0.\underline{008594402114}	& 1.\underline{897151450}236  	  		\\
			8		&	14	&	0.\underline{008594402114}	& 1.\underline{897151450}823  			\\
			\hline
			\hline
		\end{tabular}
		
		\vskip 3ex
		
		\begin{tabular}{cccc}
			\hline
			\hline
			\multicolumn{4}{c}{{\sc $\beta=0.524$}}			\hskip 20ex 								\\
			$k$	&	$p$ & $\lambda_{\max}({\mathcal A}^{{\mathcal V}_k}(\omega^{(k+1)}))$ & $\omega^{(k+1)}$ 		\\
			\hline
			4	&	6	&	0.233188266820 & 1.985354587351										\\
			5	&	8	&	-0.\underline{00}6697302959	& 1.\underline{90}1962042436  				\\
			6	&	10	&	-0.\underline{004923}289259	&1.\underline{9083}57152861					\\
			7	&	12	&	-0.\underline{004923056427}&1.\underline{90834804}5018	  				\\
			8	&	14	&	-0.\underline{004923056427} & 1.\underline{90834804}1619					\\
			\hline
			\hline
		\end{tabular}
		
	\end{center}
		
	\caption{The minimal value of the reduced eigenvalue function 
	$\lambda_{\max}({\mathcal A}^{{\mathcal V}_k}(\omega^{(k+1)}))$,
	the corresponding global minimizer 
	$\omega^{(k+1)} := \arg \min_{\omega \in \Omega} \lambda_{\max}({\mathcal A}^{{\mathcal V}_k}(\omega))$, 
	and the subspace dimensions $p := {\rm dim} \: \mathcal{V}_k$ are listed with respect to the iteration number $k$ 
	for the example concerning the hyperbolicity of a QEP.}
	\label{QEP_mult}
\end{table}

\medskip

\noindent
\textbf{Linear systems in saddle point form.} 
We consider again the matrix pair $(\mathcal{A},\mathcal{J})$ discussed in Section \ref{sec:num_exper_small}, 
where $\mathcal{A} \in {\mathbb R}^{(n+m)\times (n+m)}$ is the coefficient matrix of the form (\ref{saddle_form}) 
that originates from a discretization of the Stokes equation, and $\mathcal{J}=$diag$(I_n,-I_m)$.  
We run our subspace procedure to minimize $\lambda_{\max}(\mathcal{A}\cos\theta+\mathcal{J}\sin\theta)$ 
over $\theta \in [0,2\pi]$. The computed results coincide with the ones obtained from a direct application of
Algorithm \ref{eigopt} in Section \ref{sec:num_exper_small}. As remarked before, 
$\lambda_{\max}(\mathcal{A}\cos\theta_\ast+\mathcal{J}\sin\theta_\ast)$ 
has multiplicity three at the global minimizer $\theta_\ast$. Our subspace framework again exhibits a
quadratic convergence, which is evident from Table \ref{Saddle_mult}.

\begin{table}[tb]
	\begin{center}
		\begin{tabular}{ c| c | c | c }
			\hline
			\hline
			$k$ & $p$ & $\lambda_{\max}({\mathcal A}^{{\mathcal V}_k}({\mathcal \omega}^{(k+1)}))$ & $\omega^{(k+1)}$ \\
			\hline
			2 & 2 & -0.876669786135 & 2.023688714623 \\
			3 & 4 & -0.\underline{02}3285504705 & 3.\underline{08}8556373675  \\
			4 & 7 & -0.\underline{02222}5058342 & 3.\underline{08750}3074213\\
			5 & 10 & -0.\underline{022224901666} & 3.\underline{087502918535} \\
			6 & 13 & -0.\underline{022224901666}  & 3.\underline{087502918535}\\
			\hline
			\hline
		\end{tabular}
	\end{center}
	\caption{
	This table concerns the coefficient matrix ${\mathcal A}$ for the saddle point system arising from the Stokes
	equation, and the positive definiteness of ${\mathcal A} - \mu {\mathcal J}$ for some $\mu$, where
	$\mathcal{J} := {\rm diag}(I_n, -I_m)$. The minimal value of the reduced eigenvalue function 
	$\lambda_{\max}({\mathcal A}^{{\mathcal V}_k}({\mathcal \omega}^{(k+1)}))$, the global 
	minimizer $\omega^{(k+1)}$ and the subspace dimension $p := {\rm dim} \: \mathcal{V}_k$ are listed 
	with respect to the iteration number $k$, when the subspace procedure is applied to minimize 
	$\lambda_{\max}( {\mathcal A} \cos \omega + {\mathcal J} \sin \omega )$ over $\omega$.}
	\label{Saddle_mult}
\end{table}

\medskip

\noindent
\textbf{Performance of the Subspace Framework.} 
Finally we test the performance of Algorithm \ref{subspace_algorithm} for the computation of the inner numerical 
radius of Hermitian pairs $(A_n,B_n)$ of the form
\begin{equation}\label{exam:performance_subspace}
	A_n \;= \;\frac{C_n+C_n^\ast}{2} 
		\qquad \text{and} \qquad 
	B_n \; = \; -\ri \frac{C_n-C_n^\ast}{2},
\end{equation}
with $C_n =P_n + iR_n$, the matrix $P_n$ denoting the $n\times n$ matrix obtained from the finite difference discretization 
of the Poisson operator by employing the five-point formula, and $R_n$ denoting a random $n\times n$ sparse matrix 
generated by the Matlab command \texttt{spran(n,n,20/n)}. Table \ref{performance_subspace} lists the computed values 
of the inner numerical radius by the subspace framework, number of subspace iterations and run-times in seconds 
to reach the specified accuracy for the pairs $(A_n,B_n)$ of sizes varying between $10000$ and $90000$.
The number of subspace iterations, as well as the time to solve the reduced eigenvalue optimization problems, 
do not vary much with respect to $n$. However, the time required for the computation of the largest eigenvalue
of the full problem at every iteration increases with respect to $n$. In essence the total runtime is determined
by these large-scale eigenvalue computations for large values of $n$.

\begin{center}
	\begin{table}
		\begin{tabular}{|c||ccccc|}
			\hline
			$n$			&	$\#$ iter			&	total time		&	reduced prob	&	eigval comp		& $\zeta(A_n+ \ri B_n)$	 \\
			\hline
			\hline
			10000		&	21				&	22.62		&	 4.41			&	17.16			&	653.69	\\
			22500		&	26				&	110.61		&	 6.43			&	100.96			&	983.80	\\
			40000		&	24				&	168.44		&	5.91			&	156.91   			&	1316.77	\\
			62500		&	20				&	315.34		&	13.18		&	291.97			&	1667.44	\\
			90000		&	21				&	594.99		&	11.12		&	569.67			&	1995.49	\\
			\hline
		\end{tabular}
		\caption{ \noindent The table concerns a Hermitian matrix pair $(A_n,B_n)$ with $A_n, B_n \in {\mathbb R}^{n\times n}$ 
		defined as in (\ref{exam:performance_subspace}) in terms of $C_n = P_n + \ri R_n$, the Poisson matrix $P_n$ 
		and the sparse random matrix $R_n$. The computed values of the inner numerical radius $\zeta(A_n + \ri B_n)$
		by Algorithm \ref{subspace_algorithm}, 
		the number of subspace iterations (2nd column) and the computation times (3rd-5th columns) 
		are listed for various values of $n$. For each value of $n$, the total run-time, 
		the time spent for the solution of the reduced eigenvalue optimization problems, and
		the time spent for the large-scale eigenvalue computations in seconds
		are given in the 3rd, 4th, 5th columns, respectively. 
		}
		\label{performance_subspace} 
	\end{table}
\end{center}

\section{Conclusion}
The algorithm in \cite{Mengi2014} based on piece-wise quadratic model functions appears to be 
quite effective in dealing with global minimization problems involving a non-convex largest eigenvalue
function of a Hermitian matrix depending on one parameter. On the other hand, the subspace framework
in \cite{kangal2018} is quite effective to deal with such problems when the Hermitian matrix is large.
It accurately reduces the large dimensionality by projecting the Hermitian matrix to small subspaces
formed of eigenvectors. Here we have illustrated the efficiency of these algorithms on the computation
of the inner numerical radius. 

As a by-product, we have generalized the subspace framework of \cite{kangal2018}
to better cope with non-smoothness at the minimizer. The generalized subspace framework adds
not only the eigenvector corresponding to the largest eigenvalue, but also the eigenvectors 
corresponding to nearby eigenvalues.

We have proven rapid convergence results for both algorithms in the non-smooth case when 
the largest eigenvalue is not simple generically. The algorithm in \cite{Mengi2014} is shown to
generate a sequence $\{ \ell^{(k)} \}$ of lower bounds such that both $\{ \ell^{(2k)} \}$ and
$\{ \ell^{(2k+1)} \}$ converge to the globally smallest value of the largest eigenvalue function
at a quadratic rate. The generalized subspace framework is shown to generate a sequence of
iterates $\{ \omega^{(k)} \}$ that converge to the global minimizer at a quadratic rate.
What we witness in practice is consistent with these theoretical findings. To this end, several 
numerical results concerning the inner numerical radius computation in the non-smooth case
are reported confirming the expected in theory.


\appendix

\section{Proof of Quadratic Convergence for the Level-Set Algorithm}\label{sec:level-set_quad}

	We assume $\lambda_{\max}(H(\theta))$ is not constant near $\theta_\ast$, as otherwise $r^{(j)} = f_\ast$
	for all $j$ large enough and there is nothing to prove. Consequently, there exists an interval 
	$[\theta_\ast -\delta,\theta_\ast+\delta]$ such that $\lambda_{\max}(H(\theta))$ is strictly decreasing on 
	$[\theta_\ast -\delta,\theta_\ast]$, and is strictly increasing on $[\theta_\ast ,\theta_\ast+\delta]$.
	Furthermore, for $j$ large enough, there exists an open interval $I^{(j)}_\ell \subset [\theta_\ast-\delta ,\theta_\ast+\delta]$ 
	satisfying $f(\theta) < r^{(j)}$ for all $\theta \in I^{(j)}_\ell$ and $f(\theta)$ attains the value $r^{(j)}$
	at the endpoints of this interval, as well as $r^{(j+1)} = f(\phi^{(j+1)}_\ell)$ where $\phi^{(j+1)}_\ell$
	is the midpoint of $I^{(j)}_\ell$. We assume $I^{(j)}_\ell$ is of the form $I^{(j)}_\ell = ( \ell^{(j)}_\ell , u^{(j)}_\ell )$ 
	with $\ell^{(j)}_\ell < u^{(j)}_\ell$. If $I^{(j)}_\ell$ is not of this form, the argument below applies to
	the interval $( \ell^{(j)}_\ell , u^{(j)}_\ell + 2\pi )$

	Now we can choose $\delta > 0$ small enough if necessary so that $\lambda_{\max}(H(\theta))$ has the Taylor
	expansion
	\begin{equation}\label{eq:Tay_exp}
		\lambda_{\max}(H(\theta))=\lambda_{\max}(H(\theta_\ast)) + \beta(\theta-\theta_\ast)^{2k} + O((\theta-\theta_\ast)^{2k+1})
	\end{equation}
	for all $\theta \in [\theta_\ast-\delta ,\theta_\ast+\delta]$ where $\beta > 0$ and $k \geq 1$ due to the fact that $\theta_\ast$
	is a smooth minimizer of $\lambda_{\max}(H(\theta))$. For the sake of clarity, we only consider the case $k = 1$
	here; the arguments below generalize in a straightforward manner for $k > 1$. The eigenvalue function $\lambda_{\max}(H(\theta))$
	is one-to-one on $[\theta_\ast - \delta, \theta_\ast]$ and on $[\theta_\ast, \theta_\ast + \delta]$, which imply
	the existence of the inverse functions 
	\begin{equation*}
		\begin{split}
			\theta_- : [\lambda_{\max}(H(\theta_\ast)),\lambda_{\max}(H(\theta_\ast-\delta))]\rightarrow[\theta_\ast -\delta,\theta_\ast],	 \\
			\theta_-(\lambda) = \theta		\;\;\;	\Longleftrightarrow		\;\;\;		\lambda_{\max}(H(\theta)) = \lambda,		\hskip 14ex
		\end{split}
	\end{equation*}
	and
	\begin{equation*}
		\begin{split}
			\theta_+ : [\lambda_{\max}(H(\theta_\ast)),\lambda_{\max}(H(\theta_\ast+\delta))]\rightarrow[\theta_\ast ,\theta_\ast+\delta],	\\
			\theta_+(\lambda) = \theta		\;\;\;	\Longleftrightarrow		\;\;\;		\lambda_{\max}(H(\theta)) = \lambda.		\hskip 14ex
		\end{split}
	\end{equation*}
	These functions have Puiseux series expansions \cite[page 246]{dienes1957taylor} of the form
	\begin{eqnarray}
		\theta_-(\lambda)	&=&	\theta_\ast + c_-(\lambda - \lambda_{\max}(H(\theta_\ast)))^{1/2} + O(\lambda - \lambda_{\max}(H(\theta_\ast))) 
		\label{eq:Pui_min} \\
		\theta_+(\lambda)	&=&	\theta_\ast + c_+(\lambda - \lambda_{\max}(H(\theta_\ast)))^{1/2} + O(\lambda - \lambda_{\max}(H(\theta_\ast))) 
		\label{eq:Pui_pos}
	\end{eqnarray}
  	and satisfy $\theta_-(r^{(j)})=\ell^{(j)}$, $\theta_+(r^{(j)})=u^{(j)}$. Exploiting $\lambda_{\max}(H(\ell^{(j)}))=r^{(j)}$, we have
	\[
		r^{(j)} - \lambda_{\max}(H(\theta_\ast))	\;\; = \;\;	\beta (\ell^{(j)} - \theta_\ast)^2	+	O((\ell^{(j)} - \theta_\ast)^3)
	\]
	from the Taylor expansion (\ref{eq:Tay_exp}), and
	\[
		( \ell^{(j)} - \theta_\ast ) \;\; = \;\; c_- ( r^{(j)}  -  \lambda_{\max}(H(\theta_\ast)) )^{1/2} + O(r^{(j)} - \lambda_{\max}(H(\theta_\ast)))
	\]
	from (\ref{eq:Pui_min}), from which we infer $c_-=-\beta^{-1/2}$. In a similar way, by exploiting $\lambda_{\max}(H(u^{(j)}))=r^{(j)}$
	in equations (\ref{eq:Tay_exp}) and (\ref{eq:Pui_pos}), we obtain $c_+=\beta^{-1/2}$. Hence,
	summing up (\ref{eq:Pui_min}) and (\ref{eq:Pui_pos}) and evaluating the resulting expression at $\lambda = r^{(j)}$
	give rise to
 	\begin{equation}\label{middle_point}
  	  	\frac{\ell^{(j)}+u^{(j)}}{2}=\theta_\ast + O(r^{(j)} - \lambda_{\max}(H(\theta_\ast))). 
    	\end{equation}
    	Finally, we substitute (\ref{middle_point}) in (\ref{eq:Tay_exp}) and use $\lambda_{\max}(H((\ell^{(j)} + u^{(j)})/2)) = r^{(j+1)}$ to obtain
    	\[
    		r^{(j+1)} - \lambda_{\max}(H(\theta_\ast))  =  O((r^{(j)} - \lambda_{\max}(H(\theta_\ast)))^2)
    	\]
    	as desired.

\bibliography{largescale_DNDP}
\end{document}